\documentclass[11pt,a4paper]{article}
\usepackage[utf8]{inputenc}
\usepackage[T1]{fontenc}
\usepackage{amsmath}
\usepackage{amsthm}

\usepackage{amsfonts,mathtools}
\usepackage{hyperref,cite}
\hypersetup{
	colorlinks=true,
    linkcolor={red!50!black},
    citecolor={blue!50!black},
    urlcolor={blue!50!black},
    bookmarksopen=true,
	  bookmarksnumbered,
	  bookmarksopenlevel=2,
	  bookmarksdepth=3
}

\usepackage[capitalize]{cleveref}
\usepackage{enumerate}
\usepackage{tikz}
\usetikzlibrary{decorations}

\usetikzlibrary{shapes}
\usepackage{amssymb}
\usepackage{enumitem}
\newcommand{\G}{\mathcal{G}}
\usepackage{amsfonts}
\usepackage{amssymb}
\usepackage[margin=3cm]{geometry}
\usepackage{color}

\usepackage{float}
\usepackage{graphicx}
\allowdisplaybreaks

\newcommand{\NP}{\textsf{NP}}

\newcommand{\tw}{\mathsf{tw}}
\newcommand{\fvs}{\mathsf{fvs}}
\usepackage{xspace}

\newcommand{\cmsotwo}{$\mathsf{CMSO}_2$\xspace}

\usepackage{boxedminipage}

\usepackage{latexsym}
\usepackage{graphics}
\usepackage{amsbsy}
\usepackage{mathrsfs} 
\usepackage{ifpdf}

\usepackage{amstext} 
\newtheorem{theorem}{Theorem}[section] 
\newtheorem{lemma}[theorem]{Lemma}
\newtheorem{proposition}[theorem]{Proposition}

\newtheorem{corollary}[theorem]{Corollary} 
\newtheorem{question}[theorem]{Question} 
 
\newtheorem{observation}{Observation}
\usepackage{thm-restate}
\Crefname{question}{Question}{Questions}

\title{A tame vs.~feral dichotomy for graph classes excluding an induced minor or induced topological minor}\author{Martin Milani\v{c}\thanks{FAMNIT and IAM, University of Primorska, Koper, Slovenia. Email: \texttt{martin.milanic@upr.si}.}
 \and Nevena Piva\v{c}\thanks{FAMNIT and IAM, University of Primorska, Koper, Slovenia. Email: \texttt{nevena.pivac@iam.upr.si}.} 
}
\date{}
\begin{document}

\maketitle

\begin{abstract}
A minimal separator in a graph is an inclusion-minimal set of vertices that separates some fixed pair of nonadjacent vertices. 
A graph class is said to be tame if there exists a polynomial upper bound for the number of minimal separators of every graph in the class, and feral if it contains arbitrarily large graphs with exponentially many minimal separators.
Building on recent works of Gartland and Lokshtanov [SODA 2023] and Gajarsk{\'y}, Jaffke, Lima, Novotn{\'a}, Pilipczuk, Rz\k{a}\.zewski, and Souza [arXiv, 2022], we show that every graph class defined by a single forbidden induced minor or induced topological minor is either tame or feral, and classify the two cases.
This leads to new graph classes in which \textsc{Maximum Weight Independent Set} and many other problems are solvable in polynomial time.
We complement the classification results with polynomial-time recognition algorithms for the maximal tame graph classes appearing in the obtained classifications.


\end{abstract}

\section{Introduction}

\subsection{Background and motivations}

A \emph{minimal separator} in a (finite, simple, and undirected) graph $G$ is an inclusion-minimal set of vertices separating a fixed pair of nonadjacent vertices.
Minimal separators are an important tool in structural and algorithmic graph theory.
They are relevant in the proof of the fact that every chordal graph has a simplicial vertex~\cite{D61}, appear in a related result about the existence of moplexes in arbitrary graphs~\cite{moplex}, and are closely related to minimal triangulations, which have applications to sparse matrix computations~\cite{MinTriangulations}.
Several graph algorithms, in particular for problems related to independent sets and treewidth, are based on minimal separators and the closely related notion of potential maximal cliques (see, e.g.,~\cite{MR3763653,MR3769285,MR3909546,MR3943605,TWFillIn,MR3917975,MR3311877,MR4567533,DBLP:conf/sea2/Tamaki19}).
Minimal separators have also been studied in the recent literature from various other points of view (see, e.g.,~\cite{MR4020560,MR4456176,MR4198013,MR4245262,MR4249457,AbrishamiCDTTV22}).

Some of the most significant algorithmic applications of minimal separators are related to families of graph classes with a polynomially bounded number of minimal separators.
A graph class $\G$ is said to be \emph{tame} if there exists a polynomial $p$ such that for every graph $G\in \G$ the number of minimal separators of $G$ is bounded by $p(|V(G)|).$ 
Many graph problems that are \NP-hard in general can be solved in polynomial time when restricted to tame graph classes. 
This includes \textsc{Treewidth} and \textsc{Minimum Fill-In}, as shown by Bouchitt\'e and Todinca~\cite{TWFillIn,MR1896345}, as well as \textsc{Maximum Independent Set} and \textsc{Feedback Vertex Set}, as shown by Fomin and Villanger~\cite{DBLP:conf/stacs/FominV10}.
The latter result was significantly generalized by Fomin, Todinca, and Villanger~\cite{MR3311877} who proved the following result.
We denote by $\tw(G)$ the treewidth of a graph $G$.
For a fixed integer $t\ge 0$ and fixed \cmsotwo formula $\phi$,\footnote{In \cmsotwo logic one can use vertex, edge, and (vertex or edge) set variables, 
check vertex-edge incidence, quantify over variables,
and apply counting predicates modulo fixed integers.} consider the following computational problem, $(t,\phi)$-\textsc{Maximum Weight Induced Subgraph}:
Given a graph $G$ equipped with a vertex weight function $w\colon V(G) \to \mathbb{Q}_+$, find a set $X \subseteq V(G)$ of maximum possible weight such that $G[X] \models \phi$ and $\tw(G[X]) \leq t$, or conclude that no such set exists.

\begin{theorem}[Fomin, Todinca, and Villanger~\cite{MR3311877}]\label{thm:FTV}
For fixed integer $t\ge 0$, fixed \cmsotwo formula $\phi$, and any tame graph class $\G$, the $(t,\phi)$-\textsc{Maximum Weight Induced Subgraph} problem is solvable in polynomial time for graphs in $\G$.
\end{theorem}

The algorithm given by \Cref{thm:FTV} is \emph{robust} in the sense of Raghavan and Spinrad~\cite{MR2006100}: the algorithm works on any graph $G$ and correctly solves the problem whenever the number of minimal separators is bounded by a polynomial (which is, in particular, the case when $G\in \G$); otherwise, the algorithm correctly reports that the given graph does not belong to~$\G$.
Examples of problems captured by the above framework include \textsc{Maximum Weight Independent Set}, \textsc{Maximum Weight Induced Matching}, \textsc{Maximum Weight Induced Forest} (which is equivalent, by complementation, to \textsc{Minimum Weight Feedback Vertex Set}), \textsc{Longest Induced Path}, and many others (see~\cite{MR3311877}).

The above algorithmic results motivate the quest of identifying tame graph classes and 
classifying them when restricted to particular families of graph classes.
Known examples of tame graph classes include chordal graphs~\cite{D61}, and their generalization weakly chordal graphs~\cite{TWFillIn}, cocomparability graphs of bounded interval dimension~\cite{MR1723686}, circle graphs~\cite{kloks1996treewidth}, polygon circle graphs~\cite{Suchan}, distance-hereditary graphs~\cite{DBLP:conf/esa/KloksBMK93}, probe interval graphs~\cite{DBLP:conf/stacs/ChangKLP05}, AT-free co-AT-free graphs~\cite{EkkiKohler},
$2P_2$-free graphs~\cite{MilanicP21},
$P_4$-sparse graphs~\cite{MR2204116}, extended $P_4$-laden graphs~\cite{MR2971360}, $($theta, pyramid, prism, turtle$)$-free graphs~\cite{AbrishamiCDTTV22}, graphs with minimal separators of bounded size~\cite{DBLP:conf/wg/Skodinis99}, and intersection graphs of connected subgraphs of some subdivision of some fixed graph (also known as $H$-graphs)~\cite{MR4141534}.

In contrast, relatively few classification results regarding the tameness property when restricted to particular families of graph classes are available in the literature.
The first such systematic study was done by Milanič and Pivač~\cite{MilanicP21} who gave a dichotomy characterizing the tame graph classes within the family of hereditary graph classes defined by minimal forbidden subgraphs on at most four vertices.\footnote{A graph class is \emph{hereditary} if it is closed under vertex deletion.}
Every graph class $\G$ in this family that is not tame was shown to be \emph{feral}, that is, there exists a constant $c>1$ such that for infinitely many graphs $G\in \G$ it holds that $G$ has at least $c^{|V(G)|}$ minimal separators.\footnote{While for any tame graph class it clearly holds that it is not feral, the opposite is in general not true: there are graph classes that are neither feral nor tame, see~\cite{gajarsky2022taming}.}
The result of Milanič and Pivač was substantially generalized by Gajarsk{\'y}~et~al.~\cite{gajarsky2022taming}.
Building on previous work of Gartland and Lokshtanov~\cite{GartlandL23}, they obtained a full dichotomy of hereditary graph classes defined by a finite set of forbidden induced subgraphs into tame and feral.

The above results are restricted to graph classes that are closed under induced subgraphs.
However, outside the realm of tame graph classes, many other graph inclusion relations---for example, the minor, topological minor, subgraph, induced minor, and induced topological minor relations---have been studied in the literature and have proved important in various contexts.
In particular, dichotomy results for various properties were developed for graph classes defined by a single excluded graph with respect to one of the above relations.
Such properties include bounded clique-width~\cite{DBLP:journals/algorithmica/BelmonteOS18},  well-quasi ordering~\cite{MR1185012,MR3906632}, equivalence of bounded treewidth and bounded clique number~\cite{MR4334541}, bounded tree-independence number~\cite{DALLARD2024thirdpaper}, and polynomial-time solvability of \textsc{Graph Isomorphism}~\cite{DBLP:conf/isaac/OtachiS13}.

Motivated by this state of the art, we focus in this paper on the two remaining ``induced'' relations, the induced minor and induced topological minor relation.
A graph $H$ is said to be an \emph{induced minor} of a graph $G$ if $H$ can be obtained from $G$ by a sequence of vertex deletions and edge contractions.
A particular case of an induced minor of $G$ is an \emph{induced topological minor} of $G$, that is, a graph $H$ such that some subdivision of $H$ is an induced subgraph of $G$, in which case we also say that $G$ contains an \emph{induced subdivision of $H$}. 
These relations are more challenging to work with than the induced subgraph relation.
For example, there exist graphs $H$ such that detecting $H$ as an induced minor or an induced topological minor is \NP-complete. 
For the induced minor relation, this was shown by Fellows, Kratochv\'{\i}l, Middendorf, and Pfeiffer in 1995~\cite{MR1308575}; recently, Korhonen and Lokshtanov showed that this can happen even if $H$ is a tree~\cite{DBLP:conf/soda/KorhonenL24}.
For the induced topological minor relation, L{\'{e}}v{\^{e}}que, Lin, Maffray and Trotignon~\cite{DBLP:journals/dam/LevequeLMT09}  and Maffray, Trotignon and Vušković~\cite{zbMATH05529218} 
showed that the problem is \NP-complete if $H$ is the complete graph $K_5$ or the complete bipartite graph $K_{2,4}$, respectively. 
In general, the following two questions are widely open.

\begin{question}\label{que:recognition-H-induced-minor-free}
For which graphs $H$ there exists a polynomial-time algorithm for determining if a~given graph $G$ contains $H$ as an induced minor?
\end{question}

\begin{question}\label{que:recognition-H-induced-topological-minor-free}
For which graphs $H$ there exists a polynomial-time algorithm for determining if a~given graph $G$ contains $H$ as an induced topological minor?
\end{question}

In both cases, the problem is solvable in polynomial time if every component of $H$ is a path, since in this case it suffices to check if $H$ is present as an induced subgraph.
Furthermore, if $H$ is a graph with at most four vertices, then determining if a given graph $G$ contains $H$ as an induced minor can be done in polynomial time (see~\cite{dallard2024detecting,hartinger2017new}).
This is also the case if $H$ is the complete bipartite graph $K_{2,3}$ (see~\cite{dallard2024detecting}).
Also for the induced topological minor relation, only few polynomial cases are known (see~\cite{MR3891933,DBLP:journals/jct/ChudnovskyST13,MR3891933}).

Graph classes excluding a fixed planar graph $H$ as an induced minor are also relevant for the complexity of
\textsc{Maximum Weight Independent Set} (MWIS): Given a graph $G$ and a vertex weight function $w:V(G)\to \mathbb{Q}_+$, compute an independent set $I$ in $G$ maximizing its weight $\sum_{x\in I}w(x)$.
The problem of determining the computational complexity of MWIS in particular graph classes has been extensively studied.
In particular, the problem is known to be \NP-hard in the class of planar graphs (see~\cite{MR411240}), which implies that MWIS remains \NP-hard in graphs classes defined by a single forbidden induced minor $H$ when $H$ is nonplanar.
For the case when a planar graph $H$ is forbidden as an induced minor, Dallard, Milani{\v{c}}, and {\v{S}}torgel posed the following  question in~\cite{DALLARD2024thirdpaper}.

\begin{question}\label{que:MWIS-poly-for-planar-H}
Is MWIS solvable in polynomial time in the class of $H$-induced-minor-free graphs for every planar graph $H$?
\end{question}

This question is still open, even for the cases when $H$ is the path $P_7$ or the cycle $C_6$, but some partial results are known.
\Cref{que:MWIS-poly-for-planar-H} has an affirmative answer for all of the following graphs~$H$:  the path $P_6$ (as shown by Grzesik et al.~\cite{MR3909546}), the cycle $C_5$ (as shown by Abrishami et al.~\cite{DBLP:conf/soda/AbrishamiCPRS21}), the complete bipartite graph $K_{2,t}$ for any positive integer $t$, as well as graphs obtained from the complete graph $K_5$ by deleting either one edge or two disjoint edges (as shown by Dallard, Milani{\v{c}}, and {\v{S}}torgel~\cite{DALLARD2024thirdpaper}), and the $t$-friendship graph (that is, $t$ disjoint edges plus a vertex fully adjacent to them; as shown by Bonnet et al.~\cite{DBLP:conf/esa/BonnetDGTW23}\footnote{The paper~\cite{DBLP:conf/esa/BonnetDGTW23} solves the unweighted version of the problem, however, the methods can be  easily extended to the weighted case.}).
Quasi-polynomial-time algorithms are also known for the cases when $H$ is either a path (as shown by Gartland and Lokshtanov~\cite{MR4232071} and Pilipczuk et al.~\cite{DBLP:conf/sosa/PilipczukPR21}) or, more generally, a cycle (as shown by Gartland et al.~\cite{10.1145/3406325.3451034}), or the graph $tC_3+C_4$ for any integer $t\ge 0$ (as shown Bonnet et~al.~\cite{DBLP:conf/esa/BonnetDGTW23}).
Furthermore, Korhonen showed that for any planar graph $H$, MWIS can be solved in subexponential time in the class of $H$-induced-minor-free graphs~\cite{MR4539481}.

Recall that if $\mathcal{G}$ is a tame graph class, then MWIS and many other problems are solvable in polynomial time for graphs in $\mathcal{G}$.
This motivates the following questions.

\begin{question}\label{que:tame-H-induced-minor-free}
For which graphs $H$ is the class of graphs excluding $H$ as an induced minor tame? 
\end{question}

\begin{question}\label{que:tame-H-induced-topological-minor-free}
For which graphs $H$ is the class of graphs excluding $H$ as an induced topological minor tame? 
\end{question}

The answer to \Cref{que:tame-H-induced-minor-free} may provide further partial support towards a positive answer to \cref{que:MWIS-poly-for-planar-H}.
Let us remark that for the case of induced subgraph relation, the aforementioned dichotomy of hereditary graph classes defined by a finite set of forbidden induced subgraphs into tame and feral (see~\cite{gajarsky2022taming,GartlandL23}) implies that a graph class defined by a \emph{single} forbidden induced subgraph $H$ is tame if $H$ is an induced subgraph of the path $P_4$ or of the graph $2P_2$, and feral, otherwise.

\subsection{Our results}

We completely answer \Cref{que:tame-H-induced-minor-free,que:tame-H-induced-topological-minor-free}.
We show that every graph class defined by a single forbidden induced minor or induced topological minor is either tame or feral, and classify the two cases. 
Graphs used in our characterizations are depicted in~\cref{fig:bhd}.

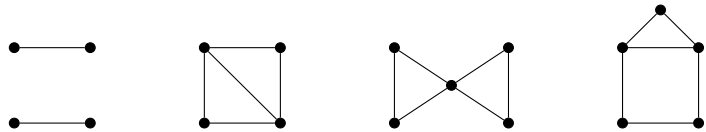
\begin{figure}[H]
\centering
\begin{tikzpicture}[vertex/.style={inner sep=1.3pt,draw,circle, fill}]
\begin{scope}[xshift=-2.5cm]
\node[vertex] (1) at (0,0) {};
\node[vertex] (2) at (1,0) {};
\node[vertex] (3) at (0,1) {};
\node[vertex] (4) at (1,1) {};
\draw[] (1)--(2);
\draw[] (3)--(4);
\end{scope}
\begin{scope}[xshift=2.5cm]
\node[vertex] (1) at (0,1) {};
\node[vertex] (2) at (0,0) {};
\node[vertex] (3) at (1.5,0) {};
\node[vertex] (4) at (1.5,1) {};
\node[vertex] (5) at (0.75,0.5) {};
\draw[] (1)--(2)--(5)--(3)--(4)--(5)--(1);
\end{scope}
\begin{scope}[xshift=5.5cm]
\node[vertex] (1) at (0,1) {};
\node[vertex] (2) at (0,0) {};
\node[vertex] (3) at (1,0) {};
\node[vertex] (4) at (1,1) {};
\node[vertex] (5) at (0.5,1.5) {};
\draw[] (1)--(2)--(3)--(4)--(1);
\draw[] (4)--(5)--(1);
\end{scope}
\begin{scope}[xshift=0cm]
\node[vertex] (1) at (0,1) {};
\node[vertex] (2) at (0,0) {};
\node[vertex] (3) at (1,0) {};
\node[vertex] (4) at (1,1) {};
\draw[] (1)--(2)--(3)--(4)--(1);
\draw[] (1)--(3);
\end{scope}
\end{tikzpicture}\caption{From left to right: the $2P_2$, the diamond, the butterfly, and the house.}\label{fig:bhd}
\end{figure}

\begin{restatable}{theorem}{dichotomyinducedminor}
\label{thm:dichotomy-induced-minor}
 Let $H$ be a graph and let $\mathcal{G}$ be the class of graphs that do not contain $H$ as an induced minor. 
 Then, the following statements are equivalent:
 \begin{enumerate}
     \item\label{cond-1} $\mathcal{G}$ is tame.
     \item\label{cond-4} $\mathcal{G}$ is not feral.
     \item\label{cond-2} $H$ is an induced subgraph of the diamond, the butterfly, or the house.
     \item\label{cond-3} $H$ is an induced minor of the butterfly or of the house.
 \end{enumerate}
\end{restatable}

\begin{restatable}{theorem}{dichotomyinducedsubdivision}
\label{thm:dichotomy-induced-subdivision}
 Let $H$ be a graph and let $\mathcal{G}$ be the class of graphs that do not contain an induced subdivision of $H$. 
Then, the following statements are equivalent:
 \begin{enumerate}
     \item\label{cond-1-subdivision} $\mathcal{G}$ is tame.
      \item\label{cond-4-subdivision} $\mathcal{G}$ is not feral.\item\label{cond-2-subdivision} $H$ is an induced subgraph of $2P_2$, the diamond, or the house. 
     \item\label{cond-3-subdivision} $H$ is an induced topological minor of $2P_2$ or of the house.
 \end{enumerate}
\end{restatable}

We complement the above results by analyzing the complexity of the recognition problems for the maximal tame graph classes in each of the above two theorems.
These correspond to~\cref{que:tame-H-induced-minor-free} for the case when $H$ is either the butterfly or the house and to \cref{que:tame-H-induced-topological-minor-free} for the case when $H$ is either $2P_2$ or the house.
As already observed, determining if a given graph $G$ contains $2P_2$ as an induced topological minor can be done in polynomial time.
Determining if a given graph $G$ contains the butterfly as an induced minor can also be done in polynomial time, using a characterization of such graphs due to Dumas and Hilaire (personal communication, 2024; for completeness, we present the argument in~\Cref{sec:recognition}). 
We provide polynomial-time algorithms for the remaining two cases.

\begin{restatable}{theorem}{houseinducedminorrecognition}
\label{thm:recognition-house-minor}
Determining if a given graph $G$ contains the house as an
induced minor can be done in polynomial time.
\end{restatable}

\begin{restatable}{theorem}
{houseinducedsubdivisionrecognition}
\label{thm:recognition-house-subdivision}
Determining if a given graph $G$ contains the house as an
induced topological minor can be done in polynomial time.
\end{restatable}

Applying \Cref{thm:FTV} to the two maximal tame graph classes identified in \Cref{thm:dichotomy-induced-minor,thm:dichotomy-induced-subdivision} yields the following algorithmic implications of our results.

\begin{corollary}\label{thm:tame-applications}
For fixed integer $t\ge 0$, fixed \cmsotwo formula $\phi$, the $(t,\phi)$-\textsc{Maximum Weight Induced Subgraph} problem is solvable by a robust polynomial-time algorithm whenever the input graph does not contain the butterfly as an induced minor or the house as an induced topological minor.
\end{corollary}

The above two graph classes generalize some graph classes extensively studied in the literature (see, e.g., \cite{MR522739,MR795937, zbMATH01289516}). 
More precisely:
\begin{itemize}
\item The class of graphs that do not contain the butterfly as an induced minor generalizes the class of $2P_2$-free graphs (hence, in particular, the classes of split graphs and complements of chordal graphs).
\item The class of graphs that do not contain the house as an induced topological minor (or even as an induced minor) is a common generalization of the classes of chordal graphs, cographs, and block-cactus graphs.
Indeed, these graph classes exclude the $C_4$, the $P_4$, and the diamond as an induced minor, respectively (for the block-cactus graphs, see~\cite{MR4334541}).
\end{itemize}

Since \textsc{Maximum Weight Independent Set} is a special case of $(t,\phi)$-\textsc{Maximum Weight Induced Subgraph}, \Cref{thm:tame-applications} provides further partial support for \Cref{que:MWIS-poly-for-planar-H}, giving an affirmative answer to the question for the cases when $H$ is either the butterfly or the house.
Let us also note that the butterfly is the $2$-friendship graph, hence, the former result also follows from the aforementioned result of Bonnet et al.~\cite{DBLP:conf/esa/BonnetDGTW23}.

\subsection{Our approach}

We briefly summarize some of the key ideas leading to our results. 
An important ingredient in our tame vs.~feral classification of graph classes excluding a single induced minor or induced topological minor is a sufficient condition for tameness due to Gajarsk{\'y}~et~al.~\cite{gajarsky2022taming}.
(As the condition is somewhat technical, we postpone the precise statement of the theorem to~\Cref{sec:sufficient-conditions}.)
Indeed, we show that if $\G$ is the class of graphs excluding the butterfly as an induced minor or the class of graphs excluding the house as an induced topological minor, then $\G$ satisfies the aforementioned sufficient condition.
While in the former case the argument is rather straightforward, the latter case requires more work. 
In particular, we make use of a recent structural result due to~Dallard~et~al.~\cite{dallard2024detecting} that played an important role in their polynomial-time algorithm for the recognition of graphs not containing the complete bipartite graph $K_{2,3}$ as an induced minor (again, we refer to~\cref{sec:sufficient-conditions} for details).

Another important step towards the proofs of our dichotomy results is a characterization of graphs that are induced minors of some short prism and some short theta, which are two specific feral graph families (see \Cref{sec:preliminaries}).
More precisely, we show that if a graph $H$ is an induced minor of some short prism and of some short theta, then $H$ is an induced subgraph of the diamond, the butterfly, or the house (see \Cref{sec:dichotomy-minor}).

Our approach towards the polynomial-time algorithms for the recognition of graphs containing the house as an induced minor or induced topological minor, or the butterfly as an induced minor, can be summarized as follows.
First, we observe that a graph contains the house as an induced topological minor if and only if it contains either an induced long-unichord or an induced long-theta  (see~\Cref{sec:recognition-house-subdivision}). 
The former can be recognized in polynomial time using the algorithm by Trotignon and Pham~\cite{trotignon2018chi}, while for the detection of a long-theta in a graph we develop a polynomial-time algorithm based on the known \emph{three-in-a-tree} algorithm by~Chudnovsky and Seymour~\cite{chudnovsky2010three} (and its recent improvement by~Lai, Lu and Thorup~\cite{DBLP:conf/stoc/LaiLT20}).

The result for graphs containing the house as an induced minor is reduced to the previous case.
 This is done by means of the following structural result (\Cref{thm:characteriation-him-free}): a graph $G$ contains the house as an induced minor if and only if $G$ contain an induced subdivision of the house or $G$ contains an induced long twin wheel, that is, a graph obtained from a cycle of length at least five by replacing a vertex with a pair of adjacent vertices with the same closed neighborhoods.

Finally, let us mention that, as several of our results deal with the case when the forbidden induced minor $H$ contains vertices of degree~$2$, we establish a property of induced minor models of such graphs $H$ that can be assumed without loss of generality 
(see \Cref{sec:thin-walks}). 

\subsection{Structure of the paper}

\Cref{sec:preliminaries} contains preliminary definitions and notations used throughout the paper. 
In~\Cref{sec:thin-walks} we discuss induced minor models of graphs containing vertices of degree~$2$. 
In~\Cref{sec:sufficient-conditions} we show that excluding the butterfly as an induced minor or the house as an induced topological minor results in tame graph classes. 
These results are used in~\Cref{sec:dichotomy-minor}, where the dichotomy theorems are developed characterizing tame graph classes among graph classes excluding a single induced minor or induced topological minor.
Finally, in~\Cref{sec:recognition} we develop polynomial-time algorithms for the recognition of classes of house-induced-minor-free graphs, butterfly-induced-minor-free graphs, and house-induced-topological-minor-free graphs.
We conclude the paper with some open questions and directions for future research in~\cref{sec:conclusion}.

\section{Preliminaries}
\label{sec:preliminaries}

The vertex set and the edge set of a graph $G$ are denoted by $V(G)$ and $E(G)$, with $n=|V(G)|$ and $m=|E(G)|$.
When discussing the running times of algorithms, we will slightly abuse the notation by writing  $\mathcal{O}(m)$ instead of $\mathcal{O}(n+m)$ for linear running time.

The \emph{neighborhood} of a vertex $v$ in a graph $G$ is the set $N_G(v)$ of all vertices adjacent to $v$ in $G$.
The \emph{closed neighborhood} of $v$ is the set $N_G(v)\cup\{v\}$, denoted by $N_G[v]$.
Given a set $X\subseteq V(G)$, we denote by $N(X)$ and $N[X]$ the sets $\bigcup_{v\in X}N_G[v]\setminus X$ and $\bigcup_{v\in X}N_G[v]$, respectively.
If the graph $G$ is clear from the context, we simply write $N(v)$, $N[v]$, $N(X)$ and $N[X]$ instead of $N_G(v)$, $N_G[v]$, $N_G(X)$ and $N_G[X]$, respectively.
A \emph{clique} in a graph $G$ is a set of pairwise adjacent vertices in $G$, while an \emph{independent set} in a graph $G$ is a set of pairwise non-adjacent vertices in $G$.
A \emph{triangle} in a graph $G$ is a clique of size $3$.
A \emph{minimal separator} in a graph $G$ is a set $S\subseteq V(G)$ such that there exist non-adjacent vertices $a,b\in V(G)$ belonging to distinct connected components of $G-S$, while for any $S'\subset S$, the vertices $a$ and $b$ are in the same connected component of $G-S'$.
A set $S$ of vertices in a graph $G$ is a \emph{dominating set} in $G$ if every vertex in $G$ is either in $S$ or has a neighbor in it. 

Given a graph $G$, a \emph{walk} in $G$ is a finite nonempty sequence $W=(w_1,\ldots, w_k)$ of vertices in $G$ such that every two consecutive vertices are joined by an edge in $G$.
If $w_1=w_k$, then $W$ is said to be a \emph{closed} walk.
The \emph{length} of a walk $(w_1,\ldots, w_k)$ is defined to be $k-1$.
Given two vertices $u,v$ in a graph $G$, a \emph{$u,v$-walk} in $G$ is any walk $(w_1,\ldots, w_k)$ in $G$ such that $u = w_1$ and $v = w_k$.
More generally, for two sets $A,B\subseteq V(G)$, \emph{a walk from $A$ to $B$} is any walk $(w_1,\ldots, w_k)$ in $G$ such that $w_1\in A$ and $w_k\in B$. 
Given two walks $W=(w_1,\ldots, w_k)$ and $Z=(z_1,\ldots, z_\ell)$ in a graph $G$ such that $w_k=z_1$, we define the \emph{concatenation} of $W$ and $Z$ to be the walk obtained by traversing first $W$ and then $Z$, that is, the walk $(w_1,\ldots, w_k=z_1,z_2,\ldots, z_\ell)$.
The concatenation of walks $W$ and $Z$ will be denoted by $W\oplus Z$. 
A \emph{path} in $G$ is a walk in which all vertices are pairwise distinct.
Given a $u,v$-path $P$, the vertices $u$ and $v$ are the \emph{endpoints} of $P$, while all the other vertices are its \emph{internal vertices}.
A \emph{cycle} in $G$ is a closed walk with length at least $3$ such that all its vertices are pairwise distinct, except that $w_1=w_k$.
The \emph{distance} between $u$ and $v$ in a connected graph $G$ is the minimum length of a $u,v$-path in $G$.

A graph $H$ is an \emph{induced subgraph} of a graph $G$ if $V(H)\subseteq V(G)$ and $E(H)=\{uv\in E(G)\colon \{u,v\}\subseteq V(H)\}$. 
In this case, the graph $H$ will also be called the \emph{subgraph of $G$ induced by $V(H)$}. 
Given a set $S \subseteq V (G)$, we denote by $G[S]$ the subgraph of $G$ induced by $S$.
The \emph{complement} of a graph $G$ is the graph $\overline{G}$ with the same vertex set as $G$ and with the edge set $E(\overline{G})=\{uv\colon \{u,v\}\subseteq V(G),u\neq v, uv\notin E(G)\}$. 
An \emph{anticomponent} of a graph $G$ is the subgraph of $G$ induced by the vertex set of a connected component of~$\overline{G}$.
A graph is \emph{anticonnected} if its complement is connected.
If $G$ is a graph and $A$ and $B$ are disjoint subsets of $V(G)$, we say that they are \emph{complete} (resp., \emph{anticomplete}) \emph{to each other} in $G$ if $\{ab\colon a\in A, b\in B\} \subseteq E(G)$ (resp., $\{ab\colon a\in A, b\in B\}\cap E(G)=\emptyset$).
If the vertex set of $G$ can be partitioned into sets $V_1$ and $V_2$ that are anticomplete to each other in $G$, then $G$ is said to be the \emph{disjoint union} of graphs $G[V_1]$ and $G[V_2]$; we denote this by $G=G[V_1]+G[V_2]$. 
Similarly, if the vertex set of a graph $G$ can be partitioned into two sets $V_1$ and $V_2$ that are complete to each other in $G$, we say that $G$ is the \emph{join} of the subgraphs of $G$ induced by $V_1$ and $V_2$; we denote this by $G=G[V_1]\ast G[V_2]$.
Given a non-negative integer $k$, the disjoint union of $k$ copies of $G$ is denoted by $kG$.

\subsection{Particular graphs and graph classes}

\begin{sloppypar}
The complete graph on $n$ vertices is denoted by $K_n$. 
We denote by $P_n$ the $n$-vertex \emph{path}, that is, a graph whose vertices $\{v_1,\ldots, v_n\}$ can be ordered linearly so that two vertices are adjacent if and only if they appear consecutively in the ordering.
Similarly, for an integer $n\ge 3$, we denote by $C_n$ the $n$-vertex \emph{cycle}, that is, a graph whose vertices $\{v_1,\ldots, v_n\}$ can be ordered cyclically so that two vertices are adjacent if and only if they appear consecutively in the ordering.
A \emph{hole} in a graph $G$ is an induced subgraph isomorphic to the $k$-vertex cycle $C_k$ for some $k\ge 4$.
A graph is \emph{bipartite} if its vertex set can be partitioned into two independent sets called \emph{parts}.
A \emph{complete bipartite} graph is a bipartite graph having all possible edges joining vertices in different parts (in other words, a join of two edgeless graphs).
A graph $K_{p,q}$ is a complete bipartite graph with parts of sizes $p$ and $q$, respectively.
A \emph{co-bipartite} graph is the complement of a bipartite graph, that is, a graph whose vertex set can be partitioned into two cliques.
A graph is \emph{acyclic} if it does not contain any cycle. 
\end{sloppypar}

\begin{sloppypar}
The \emph{diamond} is the graph obtained from the complete graph $K_4$ by deleting an edge.
The \emph{house} is the graph with five vertices $a,b,c,d,e$ and the following edges: $ab,bc,cd,de,ae,ad$.
The \emph{butterfly} is the join of $2P_2$ and $P_1$.
The \emph{gem} is the join of graphs $P_1$ and $P_4$. 
The \emph{claw} is the graph $K_{1,3}$.
The \emph{paw} is the graph obtained from the claw by adding to it one edge.
\end{sloppypar}
Given three positive integers $i,j,k$, at most one of which is equal to $1$, we denote by $\Gamma_{i,j,k}$ a graph $G$ consisting of two vertices $a$, $b$ and three paths $P$, $Q$, $R$, each from $a$ to $b$, and otherwise vertex-disjoint, such that the lengths of the paths $P$, $Q$, $R$ are $i$, $j$, and $k$, respectively, and each of the sets $V(P)\cup V(Q)$, $V(P)\cup V(R)$, and $V(Q)\cup V(R)$ induces a hole in $G$.
We will mostly be interested in the case when the vertices $a$ and $b$ are non-adjacent in $\Gamma_{i,j,k}$, that is, when $i,j,k\ge 2$.
Any such graph will be referred to as a \emph{theta}.
Given an integer $k\ge 3$, a \emph{short $k$-theta} (or simply a \emph{$k$-theta}) is a graph obtained as the union of $k$ internally disjoint paths of length $3$ with common endpoints $a$ and $b$.
More precisely, a $k$-theta is a graph $G$ with vertex set $V(G) =\{a, a_1,\ldots, a_k, b,b_1,\ldots, b_k\}$, and its set of edges consisting of the pairs of the following form: $aa_i, bb_i$, and $a_ib_i$ for $1\le i \le k$ (see~\cref{fig:structures}). 
Any graph that is a $k$-theta for some $k\ge 3$ will be referred to as a \emph{short theta}.

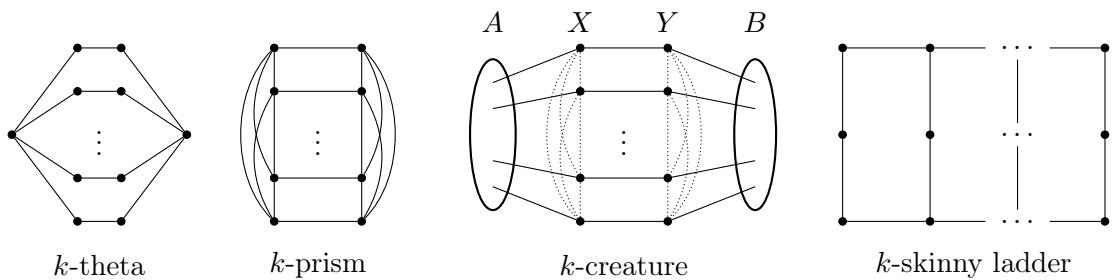
\begin{figure}[H]
\centering
\begin{tikzpicture}[scale=1.15,vertex/.style={inner sep=1pt,draw,circle, fill}]
\begin{scope}[xshift=-2cm]
\node[vertex] (a) at (0,0) {};
\node[vertex] (b) at (2,0) {};
\node[vertex] (x3) at (0.75,-0.5) {};
\node[vertex] (x4) at (0.75,-1) {};
\node[vertex] (x1) at (0.75,0.5) {};
\node[vertex] (x2) at (0.75,1) {};
\node[vertex] (y3) at (1.25,-0.5) {};
\node[vertex] (y4) at (1.25,-1) {};
\node[vertex] (y1) at (1.25,0.5) {};
\node[vertex] (y2) at (1.25,1) {};
\node[] (dots) at (1,0) {$\vdots$};
\node[] (s) at (1,-1.5) {$k$-theta};
\draw[] (a) --  (x1) -- (y1)--(b);
\draw[] (a) --  (x2) -- (y2)--(b);
\draw[] (a) --  (x3) -- (y3)--(b);
\draw[] (a) --  (x4) -- (y4)--(b);
\end{scope}

\begin{scope}[xshift=1cm]
\node[vertex] (x3) at (0,-0.5) {};
\node[vertex] (x4) at (0,-1) {};
\node[vertex] (x2) at (0,0.5) {};
\node[vertex] (x1) at (0,1) {};
\node[vertex] (y3) at (1,-0.5) {};
\node[vertex] (y4) at (1,-1) {};
\node[vertex] (y2) at (1,0.5) {};
\node[vertex] (y1) at (1,1) {};
\draw (x1) .. controls (-0.3,0.5) and (-0.3,0) .. (x3);
\draw (x2) .. controls (-0.3,0) and (-0.3,-0.5) .. (x4);
\draw (x1) .. controls (-0.5,0.5) and (-0.5,-0.5) .. (x4);

\draw (y1) .. controls (1.3,0.5) and (1.3,0) .. (y3);
\draw (y2) .. controls (1.3,0) and (1.3,-0.5) .. (y4);
\draw (y1) .. controls (1.5,0.5) and (1.5,-0.5) .. (y4);

\node[] (dots) at (0.5,0) {$\vdots$};
\node[] (s) at (0.5,-1.5) {$k$-prism};
\draw[] (y4)--(x4)--(x3)--(x2)--  (x1) -- (y1)--(y2)--(y3)--(y4);
\draw[]  (x2) -- (y2);
\draw[]  (x3) -- (y3);
\end{scope}

\begin{scope}[xshift=4.5cm]
\node[vertex] (x3) at (0,-0.5) {};
\node[vertex] (x4) at (0,-1) {};
\node[vertex] (x2) at (0,0.5) {};
\node[vertex] (x1) at (0,1) {};
\node[vertex] (y3) at (1,-0.5) {};
\node[vertex] (y4) at (1,-1) {};
\node[vertex] (y2) at (1,0.5) {};
\node[vertex] (y1) at (1,1) {};
\draw[densely dotted] (x1) .. controls (-0.3,0.5) and (-0.3,0) .. (x3);
\draw[densely dotted] (x2) .. controls (-0.3,0) and (-0.3,-0.5) .. (x4);
\draw[densely dotted] (x1) .. controls (-0.5,0.5) and (-0.5,-0.5) .. (x4);

\draw[densely dotted] (y1) .. controls (1.3,0.5) and (1.3,0) .. (y3);
\draw[densely dotted] (y2) .. controls (1.3,0) and (1.3,-0.5) .. (y4);
\draw[densely dotted] (y1) .. controls (1.5,0.5) and (1.5,-0.5) .. (y4);

\node[] (dots) at (0.5,0) {$\vdots$};
\draw[] (y4)--(x4);
\draw[densely dotted] (x4)--(x3)--(x2)--  (x1);
\draw[] (x1)-- (y1);
\draw[] [densely dotted] (y1)--(y2)--(y3)--(y4);
\draw[]  (x2) -- (y2);
\draw[]  (x3) -- (y3);
\node[] (a) at (-1,1.3) {$A$};
\node[] (a) at (0,1.3) {$X$};
\node[] (a) at (1,1.3) {$Y$};
\node[] (a) at (2,1.3) {$B$};
\node[] (s) at (0.5,-1.5) {$k$-creature};
\node [ellipse, line width=0.8pt, draw=black, minimum width=0.6cm,minimum height=2cm] at (-1,0) {};
\node [ellipse, line width=0.8pt, draw=black, minimum width=0.55cm,minimum height=2cm] at (2,0) {};
\draw[] (x1)--(-1,0.6);
\draw[] (x2)--(-1,0.3);
\draw[] (x3)--(-1,-0.3);
\draw[] (x4)--(-1,-0.6);
\draw[] (y1)--(2,0.6);
\draw[] (y2)--(2,0.3);
\draw[] (y3)--(2,-0.3);
\draw[] (y4)--(2,-0.6);
\end{scope}

\begin{scope}[xshift=7.5cm]
\node[vertex] (r1) at (0,0) {};
\node[vertex] (r2) at (1,0) {};
\node[] (rl) at (2,0) {$\ldots$};
\node[vertex] (rk) at (3,0) {};

\node[vertex] (p1) at (0,1) {};
\node[vertex] (p2) at (1,1) {};
\node[] (pl) at (2,1) {$\ldots$};
\node[vertex] (pk) at (3,1) {};
\node[] (pl1) at (1.75,1) {};
\node[] (pl2) at (2.25,1) {};

\node[vertex] (q1) at (0,-1) {};
\node[vertex] (q2) at (1,-1) {};
\node[] (ql) at (2,-1) {$\ldots$};
\node[vertex] (qk) at (3,-1) {};
\node[] (ql1) at (1.75,-1) {};
\node[] (ql2) at (2.25,-1) {};

\node[] (s) at (1.5,-1.5) {$k$-skinny ladder};
\draw[] (p1) --  (r1) -- (q1);
\draw[] (p2) --  (r2) -- (q2);
\draw[] (pl2)--(pk) --  (rk) -- (qk)--(ql2);
\draw[] (p1) --  (p2)--(pl1);
\draw[] (q1) --  (q2)--(ql1);
\draw[] (pl) --  (rl)--(ql);
\end{scope}

\end{tikzpicture}\caption{Schematic representations of the graphs $k$-theta, $k$-prism, $k$-creature, and $k$-skinny ladder. Dotted edges may or may not exist.}
\label{fig:structures}
\end{figure}

A \emph{prism} is any subdivision of $\overline{C_6}$ in which the two triangles remain unsubdivided; in particular, $\overline{C_6}$ is a prism.
For an integer $k\ge 3$, a \emph{short $k$-prism} (or simply a \emph{$k$-prism}) is a graph whose vertex set can be partitioned into two $k$-vertex cliques, say $A = \{a_1,\dots,a_k\}$ and $B = \{b_1,\dots,b_k\}$, such that for all $i,j \in \{1,\dots,k\}$, $a_i$ is adjacent to $b_j$ if and only if $i = j$. 
Any graph that is a $k$-prism for some $k\ge 3$ will be referred to as a \emph{short prism}.
The class of all short prisms and the class of all short thetas are known to be feral, which implies that they are not tame (see, e.g.,~\cite{MilanicP21, AbrishamiCDTTV22}). 
We state these observations explicitly for later use.

\begin{observation}\label{obs:tame}
\begin{enumerate}
\item\label{obs:complete-prism} The class of all short prisms is feral.
\item\label{obs:theta} The class of all short thetas is feral.
\end{enumerate}
\end{observation}

The following two graph families will also play an important role.
Given an integer $k\ge 1$, a \emph{$k$-creature} is a graph $G$ whose vertex set is a union of pairwise disjoint nonempty vertex sets $V(G)=A\cup B\cup X\cup Y$ such that (i) $A$ and $B$ induce connected subgraphs,
(ii) $A$ is anticomplete to $Y\cup B$ and $B$ is anticomplete to $A\cup X$,
(iii) every $x\in X$ has a neighbor in $A$ and every $y\in Y$ has a neighbor in $B$, 
(iv) $|X|=|Y|=k$ and $X$ and $Y$ can be enumerated as $X=\{x_1,\ldots, x_k\}$, $Y=\{y_1,\ldots, y_k\}$ so that $x_iy_i\in E(G)$ if and only if $i=j$ (see~\cref{fig:structures}).
Given an integer $k\ge 1$, a \emph{$k$-skinny ladder} is a graph consisting of two induced anticomplete paths $P=(p_1,\ldots, p_k)$, $Q=(q_1,\ldots, q_k)$, an independent set $R=\{r_1,\ldots, r_k\}$, and edges $\bigcup_{i = 1}^k\{p_ir_i, q_ir_i\}$ (see~\cref{fig:structures}).

\subsection{Graph operations and containment relations}
Given a graph $G$, its \emph{line graph} is the graph $L(G)$ with vertex set $E(G)$ in which two distinct vertices $e$ and $f$ are adjacent if and only if $e$ and $f$ have a common endpoint as edges in $G$.
If $G$ and $H$ are graphs such that no induced subgraph of $G$ is isomorphic to $H$, we say that $G$ is \emph{$H$-free}.
Given a family $\mathcal{F}$ of graphs, we say that a graph $G$ is \emph{$\mathcal{F}$-free} if no induced subgraph of $G$ is isomorphic to a member of $\mathcal{F}$.
\emph{Contracting} an edge $e = \{u,v\}$ in a graph $G$ is the operation of replacing the vertices $u$ and $v$ in $G$ with a new vertex $w$ that is adjacent precisely to vertices in $(N_G(u)\cup N_G(v))\setminus \{u,v\}$; the resulting graph is denoted by $G/e$.
A \emph{minor} of a graph $G$ is any graph obtained from $G$ by a sequence of vertex deletions, edge deletions, and edge contractions.
An \emph{induced minor} of a graph $G$ is any graph obtained from $G$ by a sequence of vertex deletions and edge contractions.
If a graph $H$ is not isomorphic to any minor of $G$, then $G$ is said to be \emph{$H$-minor-free}.
Similarly, if a graph $H$ is not isomorphic to any induced minor of $G$, then $G$ is said to be \emph{$H$-induced-minor-free}; otherwise, we say that $G$ 
\emph{contains $H$ as an induced minor}.
A \emph{subdivision} of a graph $G$ is obtained by repeated application of the operation `insert a vertex into an edge': replace the edge $uv$ by two edges $uw$ and $wv$, where $w$ is a new vertex.
An \emph{induced topological minor} of a graph $G$ is any graph $H$ such that some subdivision of $H$ is an induced subgraph of $G$.
If a graph $H$ is not isomorphic to any induced topological minor of $G$, then $G$ is said to be \emph{$H$-induced-topological-minor-free}; otherwise, we say that $G$ \emph{contains $H$ as an induced minor} or that $G$ \emph{contains an induced subdivision of $H$}. 

An \emph{induced minor model} of a graph $H$ in a graph $G$ is a collection $\{X_v\}_{v\in V(H)}$ of pairwise disjoint vertex sets $X_v\subseteq V(G)$ such that each induced subgraph $G[X_v]$ is connected, if there is an edge $uv \in E(H)$, then there is an edge between $X_u$ and $X_v$ in $G$, and for each pair of distinct and nonadjacent vertices $u,v\in V(H)$, $u\neq v$, $uv\notin E(H)$, there are no edges between $X_u$ and $X_v$.
It is not difficult to see that a graph $G$ contains a graph $H$ as an induced minor if and only if there is an induced minor model of $H$ in $G$.
If a graph $H$ is an induced minor of a graph $G$, we will sometimes refer to  the graphs $G$ and $H$ as the \emph{host graph} and the \emph{pattern graph}, respectively.

\section{Thin walks in induced minor models}
\label{sec:thin-walks}

In this section we identify a property of induced minor models that can be assumed without loss of generality whenever the pattern graph contains an edge-disjoint collection of particular walks.
A walk $W = (w_1,\ldots, w_k)$ in a graph $H$ is said to be a \emph{thin walk} if all the vertices of $W$ are pairwise distinct, except that possibly $w_1 = w_k$, and for all $i\in \{2,\ldots, k-1\}$, vertex $w_i$ has degree~$2$ in $H$.
Given a thin walk $W = (w_1,\ldots, w_k)$ in a graph $H$, vertices $w_2,\ldots, w_{k-1}$ will be referred to as the \emph{internal vertices} of $W$.

\begin{lemma}\label{lem:simple-induced-minor-models-of-paths}
Let $H$ be a graph, let $W$ be a thin walk in $H$, and let $H$ be an induced minor of a graph $G$.
Then, there exists an induced minor model $\{X_v\}_{v\in V(H)}$ of $H$ in $G$ such that $|X_{w}|=1$ for every internal vertex $w$ of $W$.
\end{lemma}

\begin{proof}
Let $\{X_v\}_{v\in V(H)}$ be an induced minor model of $H$ in $G$. 
Let $W= (w_1,\ldots, w_k)$. 
For simplicity, write $X_i$ for $X_{w_i}$ for all $i\in \{1,\ldots, k\}$.
Note that the sets $X_1,\ldots, X_k$ are nonempty and pairwise disjoint, except that $X_1 = X_k$ if $W$ is a closed walk.
If $k = 2$, there is nothing to show, so we assume that $k\ge 3$.
Consider the following particular way of constructing a walk in $G$ from $X_1$ to $X_k$.
For all $i\in \{1,\ldots, k-1\}$, let $u_iv_{i+1}$ be an edge in $G$ such that $u_i\in X_i$ and $v_{i+1}\in X_{i+1}$.
Moreover, for all $i\in \{2,\ldots, k-1\}$, let $P^i$ be a $v_i{,}u_i$-path in $G[X_i]$.
Then, the following concatenation of walks 
\[(u_1,v_2)\oplus\bigoplus_{i = 2}^{k-1} \left (P^{i}\oplus (u_{i},v_{i+1})\right)\]
is a walk in $G$ from $X_1$ to $X_k$ with all vertices pairwise distinct, except that the endpoints might coincide. 
Any walk that can be obtained by the above procedure will be referred to as a \hbox{\emph{$W$-monotone}} walk in $G$.

Let $Z=(z_1,\ldots, z_r)$ be a shortest $W$-monotone walk in $G$.
Note that the vertices $z_2,\ldots, z_{r-1}$ all belong to $\bigcup_{i=2}^{k-1}X_i$.
The minimality of $Z$ and the fact that all internal vertices of $W$ have degree~$2$ in $H$ imply that the vertices $z_2,\ldots, z_{r-1}$ have degree $2$ in the subgraph of $G$ induced by $\{z_1,\ldots, z_r\}$.
We now modify $\{X_v\}_{v\in V(H)}$ to obtain an alternative induced minor model $\{Y_v\}_{v\in V(H)}$ of $H$ in $G$, as follows:
\begin{itemize}
    \item $Y_v = X_v$ for all $v\in V(H)\setminus \{w_1,\ldots, w_{k}\}$, 
    \item $Y_{w_i}=\{z_{i}\}$ for all $i\in \{2,\ldots, k-1\}$, 
    \item $Y_{w_k}= X_k\cup \{z_i\colon k\le i< r\}$, and 
    \item if $w_1\neq w_k$, then $Y_{w_1} = X_1$.\footnote{Note that if $w_1 = w_k$, then $Y_{w_1} = Y_{w_k}$.}
\end{itemize}
Note that $|Y_{w_i}|=1$ for all $i\in \{2,\ldots, k-1\}$ and for each $v\in V(H)$, the subgraph of $G$ induced by $Y_v$ is connected.
Furthermore, the aforementioned properties of the subgraph of $G$ induced by $\{z_1,\ldots, z_r\}$ imply that $\{Y_v\}_{v\in V(H)}$ is indeed an induced minor model of $H$ in $G$.
\end{proof}

A repeated application of \Cref{lem:simple-induced-minor-models-of-paths} and its proof leads to the following.
 
\begin{proposition}
\label{prop:simple-induced-minor-models-of-all-paths}
Let $H$ be a graph, $\mathcal{W}$ be a set of edge-disjoint thin walks in $H$, and $U$ be the set of all internal vertices of walks in $\mathcal{W}$.
Let $H$ be an induced minor of a graph $G$.
Then, there exists an induced minor model $\{X_v\}_{v\in V(H)}$ of $H$ in $G$ such that $|X_u|=1$ for all $u\in U$. 
\end{proposition}

\begin{proof}
Let $\{X_v\}_{v\in V(H)}$ be an induced minor model of $H$ in $G$. 
Consider a walk $W= (w_1,\ldots, w_k)$ in $\mathcal{W}$.
Modifying the induced minor model $\{X_v\}_{v\in V(H)}$ as described in the proof of \Cref{lem:simple-induced-minor-models-of-paths} yields an induced minor model $\{Y_v\}_{v\in V(H)}$ of $H$ in $G$ such that $|Y_u|=1$ for every internal vertex $u$ in $W$ and, moreover, $\bigcup_{i = 2}^{k-1}Y_{p_i}\subseteq \bigcup_{i = 2}^{k-1}X_{p_i}$.
In particular, for any walk $W'\in \mathcal{W}$ such that $W'\neq W$ this procedure does not interfere with sets $X_u$ for internal vertices $u$ of $W$.
Thus, since no internal vertex of any walk $\mathcal{W}$ belongs to any other walk in $\mathcal{W}$, we can apply the construction to all walks in $\mathcal{W}$ in any order, resulting in an induced minor model $\{Z_v\}_{v\in V(H)}$ of $H$ in $G$ such that $|Z_u|=1$ for all $u\in U$. 
\end{proof}

\section{Sufficient conditions for tameness}
\label{sec:sufficient-conditions}

We show that the conditions regarding the forbidden induced minor or induced topological minor $H$ stated in \Cref{thm:dichotomy-induced-minor,thm:dichotomy-induced-subdivision}, respectively, are sufficient for tameness. 
More precisely, we show that the classes of graphs excluding the butterfly as an induced minor or the house as an induced topological minor are tame.
To this end, we make use of the following result by Gajarsk{\'y}~et~al.~\cite{gajarsky2022taming}.

\begin{theorem}[Theorem 3 in \cite{gajarsky2022taming}]
\label{thm:k-creature-free-tame}
    For every positive integer $k$ there exists a polynomial $p$ of degree $\mathcal{O}(k^3\cdot(8k^2)^{k+2})$ such that every graph $G$ that is $k$-creature-free and does not contain a $k$-skinny ladder as an induced minor contains at most $p(|V(G)|)$ minimal separators.
\end{theorem}

By~\Cref{thm:k-creature-free-tame}, in order to prove that a graph class $\G$ is tame, it suffices to prove that there exists an integer $k$ such that every graph in $\G$ is $k$-creature-free and does not contain $k$-skinny ladder as an induced minor.
We show that if $\G$ is the class of graphs excluding the butterfly as an induced minor or the class of graphs excluding the house as an induced topological minor, then the above condition is satisfied with $k = 3$.

We start with the easy case of graphs excluding a butterfly as an induced minor. 

\begin{lemma}
\label{lem:but-ind-minor-free-creature}
Every butterfly-induced-minor-free graph is $3$-creature-free.
\end{lemma}
\begin{proof}
Let $G$ be a butterfly-induced-minor-free graph.
Suppose for a contradiction that there exists a $3$-creature $H=(A,B,X,Y)$ in $G$, with $X=\{x_1,x_2,x_3\}$ and $Y=\{y_1,y_2,y_3\}$.
Note that each vertex in $X=\{x_1,x_2,x_3\}$ has a neighbor in $A$, and each vertex in $Y= \{y_1,y_2,y_3\}$ has a neighbor in $B$.
Contracting $A\cup B\cup \{x_2,y_2\}$ in $H$ to a single vertex $x$ gives a butterfly graph induced by vertices $x,x_1,y_1,x_3,y_3$, a contradiction.
\end{proof}

\begin{lemma}
\label{lem:but-ind-minor-free-ladder}
Let $G$ be a butterfly-induced-minor-free graph.
Then, $G$ does not contain the $3$-skinny ladder as an induced minor.
\end{lemma}

\begin{proof}
Suppose for a contradiction that the $3$-skinny ladder graph $H$ is an induced minor of $G$ and let $\{X_v\}_{v\in V(H)}$ be an induced minor model of $H$ in $G$. 
Let $p_1,p_2,p_3,q_1,q_2,q_3,r_1,r_2,r_3$ be the vertices of $H$, as in the definition of $k$-skinny ladder.
By \Cref{prop:simple-induced-minor-models-of-all-paths} we may assume that $X_{p_i}=\{p_i\}$ and $X_{q_i}=\{q_i\}$ for $i\in \{1,3\}$ and $X_{r_j}=\{r_j\}$ for all $j\in \{1,2,3\}$. 
Contracting the set $X_{q_2}\cup X_{p_2}\cup \{p_1,p_3,r_2\}$ into a single vertex $x$ results in a butterfly graph induced by $\{x,r_1,q_1,r_3,q_3\}$, a contradiction.
\end{proof}

\Cref{lem:but-ind-minor-free-creature,lem:but-ind-minor-free-ladder,thm:k-creature-free-tame} immediately imply the following.

\begin{theorem}\label{thm:butterfly-free-tame}
    The class of butterfly-induced-minor-free graphs is tame.
\end{theorem}

To derive a similar conclusion for the class of graphs excluding a house as an induced topological minor, we build on a recent work of Dallard~et~al.~\cite{dallard2024detecting} who studied the class of $K_{2,3}$-induced-minor-free graphs and obtained a polynomial-time recognition algorithm for graphs in the class. 
As part of their approach, they described a family $\mathcal{F}$ of graphs such that a graph $G$ contains $K_{2,3}$ as an induced minor if and only if $G$ contains a member of $\mathcal{F}$ as an induced subgraph.
We omit the exact description of the family $\mathcal{F}$, as we will not need it; however, a lemma from Dallard~et~al.~\cite{dallard2024detecting} used in their proof will be useful for us.

To explain the lemma, we need to introduce three more specific graph classes (see \cref{fig:S-T-M}).
Let $\mathcal{S}$ be the class of subdivisions of the claw.
The class $\mathcal{T}$ is the class of graphs that can be obtained from three paths of length at least one by selecting one endpoint of each path and adding three edges between those endpoints so as to create a triangle. 
The class $\mathcal{M}$ is the class of graphs $H$ that consist of a path $P$ and a vertex $a$, called the \emph{center} of $H$, such that $a$ is nonadjacent to the endpoints of $P$ and $a$ has at least two neighbors in $P$.
Given a graph $H \in \mathcal{S}\cup \mathcal{T}\cup \mathcal{M}$, the \emph{extremities} of $H$ are the vertices of degree one as well as the center of $H$ in case $H \in \mathcal{M}$. 
Observe that any graph $H \in \mathcal{S}\cup \mathcal{T}\cup \mathcal{M}$ has exactly three extremities.
\medskip

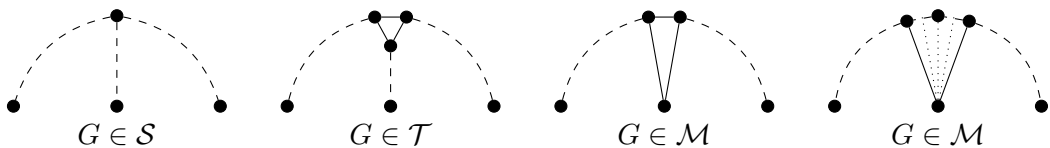
\begin{figure}[htb]
    \centering
    \begin{tikzpicture}[scale=0.8]
        \tikzset{every node/.style={draw,circle,fill=black,inner sep=0pt,minimum size=4.5pt}}
        \tikzset{decoration={amplitude=.5mm,segment length=2.25mm,post length=0mm,pre length=0mm}}
        \begin{scope}
            \node (b) at (1.7,0) {};
            \node (c) at (0,0) {};
            \node (d) at (-1.7,0) {};
            \node (a) at (0,1.5) {}; 
            \draw[dashed, bend left] (a) to (b);
            \draw[dashed] (a) to (c);
            \draw[dashed, bend right] (a) to (d);
            \node[rectangle,draw=none,fill=none] at (0,-.5) {$\strut G \in \mathcal{S}$};
        \end{scope}
        
        \begin{scope}[xshift=4.5cm]
            \node (b) at (1.7,0) {};
            \node (c) at (0,0) {};
            \node (d) at (-1.7,0) {};
            \node (ac) at (90:1) {};
            \node (ab) at (80:1.5) {};
            \node (ad) at (100:1.5) {};
            \draw (ac) to (ab)
                    (ac) to (ad)
                    (ad) to (ab);
            
            \draw[dashed, bend left] (ab) to (b);
            \draw[dashed] (ac) to (c);
            \draw[dashed, bend right] (ad) to (d);

            \node[rectangle,draw=none,fill=none] at (0,-.5) {$\strut G \in \mathcal{T}$};
        \end{scope}
        \begin{scope}[xshift=9cm]
            \node (b) at (1.7,0) {};
            \node (c) at (0,0) {};
            \node (d) at (-1.7,0) {};
            \node (ab) at (80:1.5) {};
            \node (ad) at (100:1.5) {};
            
            \draw (ab) to (ad);
            
            \draw[dashed, bend left] (ab) to (b);
            \draw (c) to (ab)
                    (c) to (ad);
            \draw[dashed, bend right] (ad) to (d);

            \node[rectangle,draw=none,fill=none] at (0,-.5) {$\strut G \in \mathcal{M}$};
        \end{scope}
        \begin{scope}[xshift=13.5cm]
            \node (b) at (1.7,0) {};
            \node (c) at (0,0) {};
            \node (d) at (-1.7,0) {};
            \node (ac) at (90:1.5) {};
            \node (ab) at (70:1.5) {};
            \node (ad) at (110:1.5) {};
            
            \draw[dashed] (ab) to (ac)
                                (ac) to (ad);
            \draw[dashed, bend left] (ab) to (b);
            \draw (c) to (ab)
                    (c) to (ad);
            \draw[dashed, bend right] (ad) to (d);
            \draw[dotted] (ac) to (c);
            \coordinate (adc) at (80:1.5);
            \coordinate (abc) at (100:1.5);
            \draw[dotted] (c) to (adc)
                        (c) to (abc);
            \node[rectangle,draw=none,fill=none] at (0,-.5) {$\strut G \in \mathcal{M}$};
        \end{scope}
    \end{tikzpicture}
    \caption{A schematic representation of graphs in $\mathcal{S}$, $\mathcal{T}$, and $\mathcal{M}$. 
    Dashed edges represent paths of positive length. 
    Dotted edges may be present or not. 
    The figure is adapted from~\cite{dallard2024detecting}.}
    \label{fig:S-T-M}
\end{figure}

\begin{lemma}[Dallard et al.~\cite{dallard2024detecting}]\label{extremities}
Let $G$ be a graph and $I$ be an independent set in $G$ with $|I| = 3$.
If there exists a connected component $C$ of $G -I$ such that $I\subseteq N(C)$, then there exists an induced subgraph $H$ of $G[N[C]]$ such
that $H \in \mathcal{S}\cup \mathcal{T}\cup \mathcal{M}$ and $I$ is exactly the set of extremities of $H$.
\end{lemma}

Using \Cref{extremities} we now derive the following.

\begin{lemma}
\label{lem:house-creature}
Let $G$ be a graph that does not contain an induced subdivision of the house.
Then, $G$ is $3$-creature-free.
\end{lemma}

\begin{proof} 
Suppose for a contradiction that there exists a $3$-creature $H=(A,B,X,Y)$ in $G$, with $X=\{x_1,x_2,x_3\}$ and $Y=\{y_1,y_2,y_3\}$.
Note that each vertex in $X=\{x_1,x_2,x_3\}$ has a neighbor in $A$.
Similarly, each vertex in $Y= \{y_1,y_2,y_3\}$ has a neighbor in $B$.
Let $Q_A$ (resp., $Q_B$) be a shortest $x_1,x_2$-path (resp., $y_1,y_2$-path) of length at least two in $G$ such that all its internal vertices belong to $A$ (resp., $B$). 
    
Assume that there exists an edge in $X$.
By symmetry, we may assume that $x_1x_2$ is an edge. 
If $y_1y_2\in E(G)$, then the vertices $\{y_1,y_2\}$ together with the vertices from the path $Q_A$ induce a subdivision of the house in $G$.
If $y_1y_2\notin E(G)$, then the vertices 
from the paths $Q_A$ and $Q_B$ induce a subdivision of the house in $G$.
Both cases lead to a contradiction, hence, $X$ is an independent set in $G$, and, by symmetry, so is $Y$.

Note that $A$ and $B$ induced connected components of $H-X$ and $H-Y$, respectively, and $X\subseteq N_H(A)$ and $Y\subseteq N_H(B)$.
Hence, by \Cref{extremities}, there exist induced subgraphs $H_A$ and $H_B$ of $H[N[A]]$ and $H[N[B]]$, respectively, such that $H_A, H_B \in \mathcal{S}\cup \mathcal{T}\cup \mathcal{M}$ and  and $X$ and $Y$ are exactly the sets of extremities of $H_A$ and $H_B$, respectively.

Assume first that one of these two graphs belongs to $\mathcal{T}$.
Up to symmetry, we may assume that $H_A\in \mathcal{T}$.
Then, the vertices $x_1,x_2,x_3$ connect to one triangle in $H_A$.  
It follows that the shortest $x_1,x_2$-path in $H_A$ together with the remaining vertex of the triangle, the edges $x_1y_1$, $x_2y_2$, and the path $Q_B$ form an induced subdivision of the house in $G$, a contradiction. 
We may therefore assume that neither $H_A$ nor $H_B$ belong to $\mathcal{T}$.
Hence, $H_A, H_B \in \mathcal{S}\cup \mathcal{M}$. 

Assume next that both graphs $H_A$ and $H_B$ belong to $\mathcal{S}.$
In this case $H$ and, hence, $G$ contains an induced subgraph of the form $\Gamma_{i,j,k}$ with $3\le i\le j\le k$, which is a subdivision of the house, yielding a contradiction.
Therefore, by symmetry, we may assume that the graph $H_A$ does not belong to $\mathcal{S}.$
It follows that $H_A\in \mathcal{M}$.  

Recall that an arbitrary graph in $\mathcal{M}$ consists of a path $P$ and a center $a$ such that $a$ is nonadjacent to the endpoints of $P$ and $a$ has at least two neighbors in $P$.   
Up to symmetry, we may assume that $x_1$ is the center in $H_A$.
Let $a$ be the neighbor of $x_1$ in $H_A$ closest to $x_2$ and let $a'$ be a neighbor of $x_1$ in the same graph such that $x_1$ has no neighbors in the interior of the path in $A$ from $a$ to $a'$. 
Let $P$ be the $x_2,a'$-path in the graph $H_A-x_1$. 
We now analyze two cases depending on the structure of $H_B$; recall that $H_B \in \mathcal{S}\cup \mathcal{M}$.

Assume first that $H_B$ belongs to $\mathcal{S}$. 
Then, the path $P$ and the $x_1,x_2$-path in $H_B$, together with the edges $x_1a,x_1a'$ form an induced subdivision of the house in $G$, a contradiction.

Consider now the case when $H_B$ belongs to $\mathcal{M}$.
Up to symmetry, there are two subcases: either $x_1$ is the center of $H_B$, or $x_2$ is the center of $H_B$. 
Assume first that $x_1$ is the center of $H_B$. 
Let $b$ be the neighbor of $x_1$ in $H_B$ that is closest to $x_2$.
Then, the path $P$ and the $x_2,b$-path in $H_B$ together with the edges $x_1a, x_1a',x_1b$ form an induced subdivision of the house in $G$, a contradiction. 
Finally, assume that $x_2$ is the center of $H_B$. 
Let $b$ be the neighbor of $x_2$ in the path $H_B-x_2$ that is closest to $x_1$. 
Then, the path $P$ and the $x_1,b$-path in the graph $H_B$ together with the edges $x_1a,x_1a',x_2b$ form an induced subdivision of the house in $G$, a contradiction.
\qedhere
\end{proof}

It remains to prove that every graph that does not contain any induced subdivision of the house excludes the $3$-skinny ladder as an induced minor. 
Recall the notation of graphs $\Gamma_{i,j,k}$ from \Cref{sec:preliminaries} and note that the $3$-skinny ladder is isomorphic to the graph $\Gamma_{2,4,4}$ (see~\cref{fig:theta322}).

\begin{figure}[H]
\centering
\begin{tikzpicture}[vertex/.style={inner sep=1.3pt,draw,circle, fill}]
\begin{scope}[xshift=3cm,yshift=1cm]
\node[vertex,label=below:{$\hat{b}$}] (a) at (0,-1) {};
\node[vertex,label=above:{$\hat{a}$}] (b) at (0,1) {};
\node[vertex,label=below:{$c$}] (3) at (-1,0) {};
\node[vertex, label=right:{$d$}] (4) at (0,0) {};
\node[vertex, label=right:{$a$}] (2) at (0.9,0.4) {};
\node[vertex, label=below:{$b$}] (1) at (0.9,-0.4) {};
\draw[] (a) --  (1) -- (2)--(b);
\draw[] (a) --  (3)--(b);
\draw[] (a) --  (4)--(b);
\end{scope}
\begin{scope}[xshift=-2cm, yshift=1cm]
\node[vertex] (r1) at (0,0) {};
\node[vertex] (r2) at (1,0) {};
\node[vertex] (r3) at (2,0) {};

\node[vertex] (p1) at (0,1) {};
\node[vertex] (p2) at (1,1) {};
\node[vertex] (p3) at (2,1) {};

\node[vertex] (q1) at (0,-1) {};
\node[vertex] (q2) at (1,-1) {};
\node[vertex] (q3) at (2,-1) {};

\draw[] (p1) --  (r1) -- (q1);
\draw[] (p2) --  (r2) -- (q2);
\draw[] (p3) --  (r3)--(q3);
\draw[] (q1) --  (q2)--(q3);
\draw[] (p1) --  (p2)--(p3);
\end{scope}
\end{tikzpicture}
\caption{The graphs $\Gamma_{2,4,4}$ (left) and $\Gamma_{2,2,3}$ (right).}
\label{fig:theta322}
\end{figure}
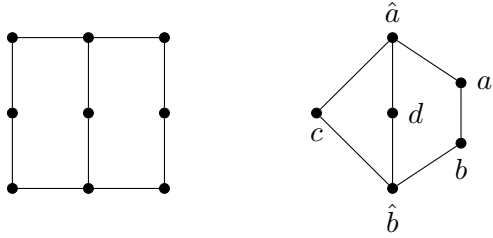

Note also that the graph $\Gamma_{2,2,3}$ is an induced minor of the graph $\Gamma_{2,4,4}$.
Hence, in order to show that every graph that does not contain any induced subdivision of the house excludes the $3$-skinny ladder as an induced minor, it suffices to show the following.

\begin{lemma}\label{lem:house-theta}
Let $G$ be a graph that does not contain an induced subdivision of the house.
Then, $G$ is $\Gamma_{2,2,3}$-induced-minor-free.
\end{lemma}

\begin{proof}
Suppose for a contradiction that $G$ contains $\Gamma_{2,2,3}$ as an induced minor.
Note that the three paths in $\Gamma_{2,2,3}$ joining the two vertices of degree three form a collection of edge-disjoint thin walks in $\Gamma_{2,2,3}$.
Hence, using the notation from \cref{fig:theta322}, we may assume by \Cref{prop:simple-induced-minor-models-of-all-paths} that there exists an induced minor model $\{X_v\}_{v\in V(\Gamma_{2,2,3})}$ of $\Gamma_{2,2,3}$ in $G$ such that $X_u=\{u\}$ for all $u\in \{a,b,c,d\}$. 
Denoting by $H$ the subgraph of $G$ induced by $\{X_v\}_{v\in V(\Gamma_{2,2,3})}$, we have that $V(H)=\{a,b,c,d\}\cup A\cup B$, where both $A$ and $B$ induce connected subgraphs of $H$, vertices $c$ and $d$ have neighbors in both $A$ and $B$, vertex $a$ has neighbors in $A$ but not in $B$, vertex $b$ has neighbors in $B$ but not in $A$, and $ab$ is an edge in $H$, and the sets $\{a,c,d\}$ and $\{b,c,d\}$ are independent in~$H$. 

Note that $A$ is a connected component of $H-S_A$ such that each vertex in $S_A = \{a,c,d\}$ has a neighbor in $A$.
Similarly, $B$ in a connected component of $H-S_B$ such that each vertex in $S_B = \{b,c,d\}$ has a neighbor in $B$.
By \Cref{extremities}, there exist induced subgraph $H_A$ of $H[N[A]]$ and $H_B$ of $H[N[B]]$, such that $H_A,H_B\in  \mathcal{S}\cup\mathcal{T}\cup \mathcal{M}$ and $S_A$ and $S_B$ are exactly the sets of extremities of $H_A$ and $H_B$, respectively.

Assume first that one of these two graphs belongs to $\mathcal{T}$. 
Up to symmetry, we may assume that $H_A\in \mathcal{T}$.
Then, the vertices $a,c,d$ connect to one triangle in $A$. 
Let $P$ be a shortest $c,d$-path with interior in $B$. 
It follows that the shortest $c,d$-path in $H_A$, together with the remaining vertex of the triangle and the path $P$ induces a subdivision of the house in $G$. 
We may therefore assume that neither $H_A$ nor $H_B$ belong to $\mathcal{T}$.
Hence, $H_A, H_B\in \mathcal{S}\cup \mathcal{M}.$

Assume next that both of these graphs belong to $\mathcal{S}$.  
In this case $G$ contains an induced subgraph of the form $\Gamma_{i,j,k}$ with $2\le i\le j\le k$ such that $k\ge 3$, which is a subdivision of the house.
This yields a contradiction, and by symmetry, we may assume that $H_A\in \mathcal{M}$.  

In the rest of the proof, we will consider the graph $H'$ obtained from $H$ by contracting the edge $ab$. 
Denoting by $u$ the contracted vertex, observe that the set $S = \{u,c,d\}$ is an independent set in $H'$.
Furthermore, $A$ and $B$ are connected components of $H'-S$ and each vertex of $S$ is adjacent in $H'$ to some vertex in each of $A$ and $B$.
Hence, by \Cref{extremities}, there exist induced subgraphs $H_A'$ and $H_B'$ of $H'[N[A]]$ and $H'[N[B]]$, respectively, such that $H_A', H_B'\in \mathcal{S}\cup\mathcal{T}\cup \mathcal{M}$ and $S$ is exactly the sets of extremities of $H_A'$ and $H_B'$.

By analyzing various cases regarding the structure of these two graphs, we next show that $H'$ contains an induced subdivision of the house.
Since the graph $H$ is obtained from the graph $H'$ by subdividing an edge, this will imply that $H$, and hence also $G$, contains an induced subdivision of the house, a contradiction.

Note first that, since the graph $H'[N[A]]$ is isomorphic to the graph $H[N[A]]$, it follows that $H_A'\in \mathcal{M}$.  
We may assume that $u$ is the center in $H_A'$.
The graph $H_B'$ either belongs to $\mathcal{S}$ or to $\mathcal{M}$.
Let $x$ be the neighbor of $u$ in the graph $H_A'$ closest to $c$ and let $x'$ be a neighbor of $u$ in the same graph such that $u$ has no neighbors in the interior of the path in $A$ from $x$ to $x'$. 
Let $P$ be a $c,x'$-path in the graph $H_A'-u$. We consider three cases.

Assume first that $H_B'\in \mathcal{S}$.
 Then the path $P$ and the shortest $c,u$-path in the graph $H_B'$, together with the edges $ux, ux'$ form a subdivision of the house in the graph $H'$, a contradiction.

Consider now the case when $H_B'\in \mathcal{M}$ and $u$ is a center of $H_B'$. 
Let $y$ be the neighbor of $u$ in the graph $H_B'$ closest to $c$. 
Then the path $P$ and the $c,y$-path in the graph $H_B'$ together with the edges $ux,ux',uy$ form a subdivision of the house in the graph $H'$, a contradiction. 
    
Finally, assume that $H_B'\in \mathcal{M}$ and $c$ is the center of $H_B'$. 
Let $y$ be the neighbor of $c$ in the path $H'_B-c$ closest to $u$. 
Then, the path $P$ and the $u,y$-path in the graph $H_B'$ together with the edges $ux,ux',cy$ form a subdivision of the house in the graph $H'$, a contradiction.
\end{proof}

\begin{theorem}\label{thm:house-itm-free-tame}
    The class of graphs that do not contain an induced subdivision of the house is tame.
\end{theorem}
\begin{proof}
    Let $\mathcal{G}$ be a class of graphs that do not contain an induced subdivision of the house and let $G$ be a graph in $\mathcal{G}$. 
    By \Cref{lem:house-creature}, $G$ is $3$-creature-free.
    By \cref{lem:house-theta}, $G$ is $\Gamma_{2,2,3}$-induced-minor-free.
    Since the graph $\Gamma_{2,2,3}$ is an induced minor of the graph $\Gamma_{2,4,4}$, which is isomorphic to the $3$-skinny ladder, we conclude that $G$ does not contain a $3$-skinny ladder as an induced minor.
    By \Cref{thm:k-creature-free-tame}, $\G$ is tame.
\end{proof}

We state the following corollary for later use. 

\begin{corollary}\label{cor:house-im-free}
    The class of house-induced-minor-free graphs is tame.
\end{corollary}
\begin{proof}
Immediate from \Cref{thm:house-itm-free-tame}, since the class of house-induced-minor-free graphs is a subclass of the class of graphs that do not contain any induced subdivision of the house. 
\end{proof}

\section{Dichotomy results}\label{sec:dichotomy-minor}

In this section we prove~\cref{thm:dichotomy-induced-minor,thm:dichotomy-induced-subdivision}, characterizing tame graph classes among graph classes excluding a single graph as an induced minor or as an induced topological minor.

We start with a lemma characterizing graphs that are simultaneously induced minors of both some short prism and some short theta.
In order to do that, we need a preparatory technical lemma regarding graph parameters that, informally speaking, measure the distance to minor-closed classes.
Let $\mathcal{F}$ be a family of graphs.
For a graph $G$, we define the following parameter:
$c_{\mathcal{F}}(G)=\min \{|S| \colon G-S$ is $\mathcal{F}$-minor-free$\}$ (see, e.g.,~\cite{MR4229413}).

\begin{lemma}\label{lem:fvn-monotonicity-parameter}
Let $\mathcal{F}$ be a family of graphs and let $G$ and $H$ be graphs such that $H$ is a minor of $G$.
Then, $c_{\mathcal{F}}(H)\le c_{\mathcal{F}}(G)$.
\end{lemma}

\begin{proof}
Let $k = c_{\mathcal{F}}(G)$.
It suffices to show that any graph $G'$ obtained from $G$ by a single vertex deletion, edge deletion, or edge contraction satisfies $c_{\mathcal{F}}(G')\le k$.
Let $S\subseteq V(G)$ be a set of size $k$ such that $G-S$ is $\mathcal{F}$-minor-free.

If $G' = G-v$ for some vertex $v\in V(G)$, then $S' = S\cap V(G')$ is a set of size at most $k$ such that $G'-S'$ is an induced subgraph of $G-S$ and, hence, $\mathcal{F}$-minor-free.

If $G' = G-e$ for some edge $e\in E(G)$, then $G'-S$ is a subgraph of $G-S$ and, hence, $\mathcal{F}$-minor-free.

Finally, assume that $G' = G/e$ for some edge $e= uv\in E(G)$, and let $w$ be the new vertex.
If at least one of $u$ and $v$ belongs to $S$, then $S' = (S\setminus \{u,v\})\cup \{w\}$ is a set of size at most $k$ such that $G'-S'$ is an induced subgraph of $G-S$ and, hence, $\mathcal{F}$-minor-free.
Otherwise, $u,v\in V(G)\setminus S$, in which case $G'-S$ is a minor of $G-S$ and, hence, $\mathcal{F}$-minor-free.
\end{proof}

We need the following special case of \cref{lem:fvn-monotonicity-parameter}.
A \emph{feedback vertex set} in a graph $G$ is a set $S\subseteq V(G)$ such that $G-S$ is acyclic.
The \emph{feedback vertex set number} of $G$, denoted by $\fvs(G)$, is the minimum cardinality of a feedback vertex set.
Since a graph is acyclic if and only if it does not contain the cycle $C_3$ as a minor, we have $\fvs(G) = c_{\mathcal{F}}(G)$ for $\mathcal{F} = \{C_3\}$; hence, the following holds.

\begin{corollary}\label{cor:fvn}
If a graph $H$ is a minor of a graph $G$, then $\fvs(H)\le\fvs(G)$.
\end{corollary}

This result has the following consequence.

\begin{corollary}\label{cor:fvn-theta-minor}
Let $H$ be a graph that is a minor of some short theta. 
Then, $\fvs(H)\le 1$.
\end{corollary}

\begin{proof}
Let $k$ be a positive integer, let $G$ be the $k$-theta and let $H$ be a minor of $G$.
Since deleting a vertex of degree $k$ from $G$ results in an acyclic graph, we have $\fvs(G)\le 1$.
Hence, \cref{cor:fvn} implies that $\fvs(H)\le \fvs(G)\le 1$.
\end{proof}

We can now prove the announced result, a characterization of graphs that are induced minors of some short prism, as well as of some short theta.
It turns out that there are only finitely many such graphs.
Some of them are depicted in~\cref{fig:small-induced-subgraphs-diamond-house-2p2}.
\medskip

\begin{figure}[H]
\centering
\begin{tikzpicture}[scale=1.15,vertex/.style={inner sep=1.3pt,draw,circle, fill}]
\begin{scope}[xshift=-0.5cm]
\node[vertex] (1) at (0.2,1) {};
\node[vertex] (2) at (0.2,0) {};
\node[vertex] (3) at (0.7,0.5) {};
\draw[] (1)--(2);
\node[] (s) at (0.5,-0.5) {$P_2+P_1$};
\end{scope}
\begin{scope}[xshift=1.65cm]
\node[vertex] (1) at (0.5,1) {};
\node[vertex] (2) at (0,0) {};
\node[vertex] (3) at (1,0) {};
\draw[] (1)--(2)--(3)--(1);
\node[] (s) at (0.5,-0.5) {$K_3$};
\end{scope}
\begin{scope}[xshift=3.8cm]
\node[vertex] (1) at (0,1) {};
\node[vertex] (2) at (0,0) {};
\node[vertex] (3) at (1,0) {};
\node[vertex] (4) at (1,1) {};
\draw[] (1)--(4);
\draw[] (3)--(2);
\node[] (s) at (0.5,-0.5) {$2P_2$};
\end{scope}
\begin{scope}[yshift=-2.5cm,xshift=-0.5cm]
\node[vertex] (1) at (0,1) {};
\node[vertex] (2) at (0,0) {};
\node[vertex] (3) at (1,0) {};
\node[vertex] (4) at (1,1) {};
\draw[] (1)--(2)--(3)--(4)--(1);
\node[] (s) at (0.5,-0.5) {$C_4$};
\end{scope}
\begin{scope}[xshift=5.9cm]
\node[vertex] (1) at (0,1) {};
\node[vertex] (2) at (0,0) {};
\node[vertex] (3) at (0.75,0.5) {};
\node[vertex] (4) at (1.5,0.5) {};
\draw[] (1)--(2)--(3)--(1);
\draw[] (4)--(3);
\node[] (s) at (0.75,-0.5) {paw};
\end{scope}
\begin{scope}[yshift=-2.5cm,xshift=1.65cm]
\node[vertex] (1) at (0,1) {};
\node[vertex] (2) at (0,0) {};
\node[vertex] (3) at (1,0) {};
\node[vertex] (4) at (1,1) {};
\draw[] (1)--(2)--(3)--(4)--(1);
\node[] (s) at (0.5,-0.5) {diamond};
\draw[] (1)--(3);
\end{scope}
\begin{scope}[yshift=-2.5cm,xshift=3.8cm]
\node[vertex] (1) at (0,1) {};
\node[vertex] (2) at (0,0) {};
\node[vertex] (3) at (1,0) {};
\node[vertex] (4) at (1,1) {};
\node[vertex] (5) at (0.5,1.5) {};
\draw[] (1)--(2)--(3)--(4)--(1);
\node[] (s) at (0.5,-0.5) {house};
\draw[] (4)--(5)--(1);
\end{scope}
\begin{scope}[xshift=5.9cm, yshift=-2.5cm]
\node[vertex] (1) at (0,1) {};
\node[vertex] (2) at (0,0) {};
\node[vertex] (3) at (1.5,0) {};
\node[vertex] (4) at (1.5,1) {};
\node[vertex] (5) at (0.75,0.5) {};
\draw[] (1)--(2)--(5)--(3)--(4)--(5)--(1);
\node[] (s) at (0.75,-0.5) {butterfly};
\end{scope}
\begin{scope}[xshift=2cm]
\end{scope}
\end{tikzpicture}\caption{Some small graphs.}\label{fig:small-induced-subgraphs-diamond-house-2p2}
\end{figure}
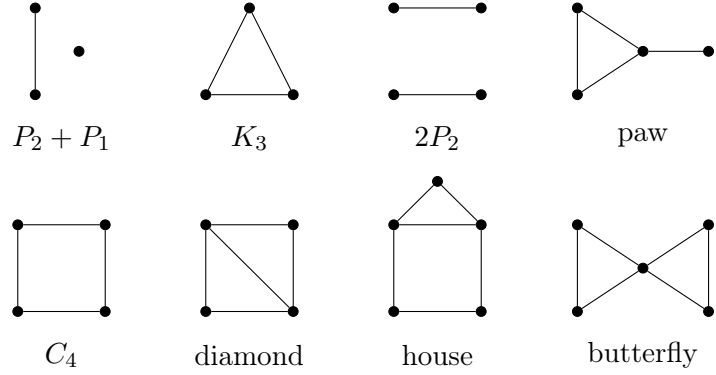

\begin{lemma}
\label{H-induced-minor-of-three-graphs}
Let $H$ be a graph that is an induced minor of some short prism and of some short theta.
Then $H$ is an induced subgraph of the diamond, the butterfly, or the house.
\end{lemma}

\begin{proof}
Let $k\ge 3$ and $\ell\ge 3$ be integers and let $H$ be a graph that is an induced minor of the $k$-prism and of the $\ell$-theta. 
Let $G_1$ be the $k$-prism and let $G_2$ be the $\ell$-theta.  
We prove a few claims about the structure of $H$.
\medskip

\noindent \textbf{Claim 1:} $H$ is $K_4$-minor-free.
\smallskip

\noindent \emph{Proof of Claim 1.}
Since $\fvs(K_4) = 2$ while $\fvs(G_2)\le 1$, \cref{cor:fvn-theta-minor} implies that $G_2$ is $K_4$-minor-free. 
But then so is $H$.\hfill$\blacklozenge$

\medskip
\noindent \textbf{Claim 2:} $H$ is $(K_3+P_1)$-free.
\medskip

\noindent \emph{Proof of Claim 2.}
It suffices to show that $G_2$ is $(K_3+P_1)$-induced-minor-free.
Suppose for a contradiction that $\widetilde{H}=K_3+P_1$ is an induced minor of $G_2$.
Label the vertices of $\widetilde{H}$ by $p,q,r,s$ so that $s$ is the isolated vertex and the vertices $p,q,r$ form a triangle.
Let $\{X_p,X_q,X_r,X_s\}$ be an induced minor model of $\widetilde{H}$ in $G_2$.
Any two of the sets $X_p$, $X_q$, and $X_r$ are connected by an edge in $G$; using one such edge per pair and, for each of the sets $X_p$, $X_q$, and $X_r$, a path connecting the endpoints of those edges belonging to the set, we obtain a cycle $C$ in $G_2$.
Since every vertex of $G_2$ either belongs to $C$ or has a neighbor on it, the set $X_s$ cannot be anticomplete to $X_p\cup X_q\cup X_r$; a contradiction with the fact that $\{X_p,X_q,X_r,X_s\}$ is an induced minor model of $\widetilde{H}$ in $G_2$.\hfill$\blacklozenge$

\medskip

\noindent \textbf{Claim 3:} $H$ is gem-free.
\medskip

\noindent \emph{Proof of Claim 3.}
It suffices to show that $G_2$ is gem-induced-minor-free.
Suppose for a contradiction that the gem $\widetilde{H}$ is an induced minor of $G_2$.
Label the vertices of $\widetilde{H}$ by $v_1,v_2,v_3,v_4,v_5$ so that  $(v_1,v_2,v_3,v_4)$ is an induced $P_4$ and $v_5$ is universal in $\widetilde{H}$. 
Let $\{X_v\}_{v\in V(\widetilde{H})}$ be an induced minor model of $\widetilde{H}$ in $G_2$. 
Note that $G_2$ contains two vertices $a$ and $b$ that belong to every cycle.
We claim that $\{a,b\}\subseteq X_{v_5}$.
This is true, for if up to symmetry $a\not\in X_{v_5}$, then there exists some $i\in \{1,2,3\}$ such that $a\not\in X_{v_i}\cup X_{v_{i+1}}\cup X_{v_5}$ and, hence, the subgraph of $G_2$ induced by $X_{v_i}\cup X_{v_{i+1}}\cup X_{v_5}$ contains a cycle not containing $a$, a contradiction.
Since $\{a,b\}\subseteq X_{v_5}$, the subgraph of $G_2$ induced by $\bigcup_{i=1}^4 X_{v_i}$ is an induced subgraph of $G_2-\{a,b\}$, that is, of $kP_2$.
But there is a path of length at least $3$ in $\bigcup_{i=1}^4 X_{v_i}$, which cannot be contained in $kP_2$, a contradiction.
\hfill$\blacklozenge$

\medskip
We now use Claims 1--3 to analyze the structure of $H$, showing that $H$ must be isomorphic to an induced subgraph of the diamond, the butterfly, or the house.
Note that any such graph is either a path $P_n$ or an edgeless graph $nP_1$ for some $n\le 4$, or one of the graphs depicted in~\cref{fig:small-induced-subgraphs-diamond-house-2p2}.

Since $G_1$ is co-bipartite and $H$ can be obtained by contracting edges of an induced subgraph of $G_1$, we infer that $H$ is also co-bipartite. 
Then the vertex set of $H$ is a union of two cliques $A$ and $B$. 
Since $H$ is $K_4$-free, 
each of $A$ and $B$ has at most $3$ vertices.
Furthermore, By \Cref{cor:fvn} it follows that $\fvs(H)\le 1$ and so at most one of the cliques $A$ and $B$ contains a cycle. 
Hence, $H$ has at most five vertices. 
We may assume without loss of generality that $|A|\ge |B|$.

If $A$ has a single vertex, then $H$ is isomorphic to either $P_1$, $2P_1$, or $P_2$.
If $A$ has two vertices and $B$ has at most one vertex, then $H$ is isomorphic to either $P_2$, $P_3$, or $K_3$.
If both $A$ and $B$ have two vertices, let $A=\{a_1,a_2\}$ and $B=\{b_1,b_2\}$. 
Since $H$ is $K_4$-minor-free, $A$ and $B$ are not complete to each other. 
Up to symmetry, let $a_1b_1\notin E(H)$. 
Depending on the existence of other edges between $A$ and $B$ we get that $H$ is isomorphic to either $2P_2$, $P_4$, $C_4$, the paw, or the diamond.

Consider now the case when $A=\{a_1,a_2,a_3\}$. 
If $B$ is empty, then $H\cong K_3$. 
If $B$ consists of a single vertex, then this vertex has either one or two neighbors in $A$, since $H$ is $\{K_4,K_3+P_1\}$-free.
In the first case, it follows that $H$ is a paw. 
In the second case $H$ is a diamond. 
Finally, assume that $B$ consists of two vertices $b_1$ and $b_2$. 
If $b_1$ and $b_2$ have a common neighbor in $A$, we may assume that they are both adjacent to the vertex $a_1$. 
If $H$ has no other edges, then $H$ is a butterfly.
If, up to symmetry, $a_2b_1\in E(H)$, then $a_2b_2\notin E(H)$ and $a_3b_1\notin E(H)$ since $H$ is $K_4$-free.
Since $H$ is not isomorphic to the gem, we have that $a_3b_2\in E(H)$, and, hence, $\fvs(H) = 2$, which is in contradiction with \cref{cor:fvn-theta-minor}.

If $b_1$ and $b_2$ have no common neighbors in $A$, up to symmetry we may assume that $a_1b_1\in E(H)$ and $a_2b_2\in E(H)$, while $a_1b_2$ and $a_2b_1$ are non-edges in $H$. 
Since $H$ is $K_4$-minor-free, neither of the two vertices in $B$ is adjacent to $a_3$, hence, $H$ is a house. 
\end{proof}

We now have everything ready to prove~\cref{thm:dichotomy-induced-minor,thm:dichotomy-induced-subdivision}, which we restate for the convenience of the reader.

\dichotomyinducedminor*

\begin{proof}
Clearly, \eqref{cond-1} implies \eqref{cond-4}.
We next prove that \eqref{cond-4} implies \eqref{cond-2}.
Assume that $\G$ is not feral.
By \Cref{obs:tame}, neither the class of short prisms nor the class of short thetas is a subclass of $\mathcal{G}$.
Hence, there exist integers $k\ge 3$, $\ell\ge 3$ such that $H$ is an induced minor of the $k$-prism and of the $\ell$-theta.
By \Cref{H-induced-minor-of-three-graphs}, it follows that $H$ is an induced subgraph of the diamond, the butterfly, or the house.

Since the diamond is an induced minor of the house, \eqref{cond-2} implies \eqref{cond-3}. 

It remains to prove that \eqref{cond-3} implies \eqref{cond-1}.
If $H$ is an induced minor of the butterfly, then $\G$ is a subclass of the class of butterfly-induced-minor-free graphs and $\G$ is tame by \Cref{thm:butterfly-free-tame}.
If $H$ is an induced minor of the house, then $\G$ is a subclass of the class of the class of house-induced-minor-free graphs and $\G$ is tame by \Cref{cor:house-im-free}.
\end{proof}

\dichotomyinducedsubdivision*
\begin{proof}
Clearly, \eqref{cond-1-subdivision} implies \eqref{cond-4-subdivision}.
We next show that \eqref{cond-4-subdivision} implies \eqref{cond-2-subdivision}.
Let $\mathcal{G}$ be non-feral.
By \Cref{obs:tame}, neither the class of short prisms nor the class of short thetas is a subclass of $\mathcal{G}$.
Hence, there exist integers $k\ge 3$, $\ell\ge 3$ such that the $k$-prism $G_1$ and the $\ell$-theta $G_2$ both contain an induced subdivision of $H$.
In other words, $H$ is an induced topological minor of both $G_1$ and $G_2$, and in particular, an induced minor of both $G_1$ and $G_2$.
It follows from \cref{H-induced-minor-of-three-graphs} that $H$ is an induced subgraph of the diamond, the butterfly, or the house.
Note that every proper induced subgraph of the butterfly is also an induced subgraph of either $2P_2$, the diamond, or the house (see~\cref{fig:small-induced-subgraphs-diamond-house-2p2}).
Hence, to complete the proof of the implication, it suffices to show that $H$ is not a butterfly.

Suppose for a contradiction that $H$ is the butterfly. 
Then, $G_2$ contains an induced subdivision of the butterfly. 
Let $F$ be an induced subdivision of the butterfly contained in $G_2$ and let $x$ be the vertex of degree $4$ in $F$ (all other vertices in $F$ are of degree $2$). 
Let $a$ and $b$ be the two vertices of degree $\ell$ in $G_2$.
Every induced cycle in $G_2$ (and thus in $F$) contains both $a$ and $b$, so any two induced cycles in $G_2$ have at least two common vertices.
But the two cycles of $F$ intersect only in $x$, a contradiction.

Since the diamond is an induced topological minor of the house, \eqref{cond-2-subdivision} implies \eqref{cond-3-subdivision}. 

It remains to prove that \eqref{cond-3-subdivision} implies \eqref{cond-1-subdivision}.
If $H$ is an induced topological minor of $2P_2$, then $\G$ is a subclass of the class of $2P_2$-free graphs. 
The class of $2P_2$-free graphs is tame (see, e.g.,~\cite{MilanicP21}), and we infer that the class $\G$ is tame as well. 
If $H$ is an induced topological minor of the house, then the class $\G$ is a subclass of the class of graphs that do not contain any induced subdivision of the house, and $\G$ is tame by \Cref{thm:house-itm-free-tame}.
\end{proof}

\section{Recognition algorithms}
\label{sec:recognition} 

In this section we give polynomial-time algorithms for the recognition of maximal tame graph classes appearing in our dichotomy theorems \Cref{thm:dichotomy-induced-minor,thm:dichotomy-induced-subdivision}, except for the class of $2P_2$-free graphs, which we already discussed in the introduction.
We first prove \cref{thm:recognition-house-subdivision}, dealing with graphs containing the house as an induced topological minor, in \Cref{sec:recognition-house-subdivision}.
The result for graphs containing the house as an induced minor and developed in \Cref{sec:recognition-house-minor} by a reduction to the induced topological minor case.
Finally, we explain in \Cref{subsec:butterfly} a result due to Dumas and Hilaire (personal communication, 2024) leading to a polynomial-time recognition algorithm for determining if a given graph contains the butterfly as an induced minor.

\subsection{Proof of~\cref{thm:recognition-house-subdivision}}
\label{sec:recognition-house-subdivision}

In order to recognize graphs excluding the house as an induced topological minor, we first characterize this graph class in terms of two families of forbidden induced subgraphs.
Then, we apply the three-in-a-tree algorithm by Chudnovsky and Seymour~\cite{chudnovsky2010three} and a result due to Trotignon and Pham~\cite{trotignon2018chi}.

Building on the terminology of Trotignon and Pham~\cite{trotignon2018chi}, we say that a \textit{long unichord} in a graph is an edge that is the unique chord of some cycle of length at least $5$.
A graph is \emph{long-unichord-free} if it does not contain any long unichord.
A \textit{long theta} is any theta graph other than $K_{2,3}$, that is, a graph of the form $\Gamma_{i,j,k}$ with \hbox{$\min\{i,j,k\}\ge 2$} and $\max\{i,j,k\}\ge 3$.
A graph is \emph{long-theta-free} if it does not contain any long theta as an induced subgraph. 
These concepts lead to the following characterization of the class of graphs not containing any induced subdivision of the house.

\begin{lemma}\label{thm:forbidden-subgraphs-subdivision-of-house}
A graph $G$ does not contain any induced subdivision of the house if and only if $G$ is long-unichord-free and long-theta-free.
\end{lemma}

\begin{proof}
Note that if a graph $G$ contains a long unichord $ab$, then $G$ contains a cycle $C$ with length at least $5$ such that $ab$ is the unique chord of $C$.
In this case, the subgraph of $G$ induced by $V(C)$ is isomorphic to a graph of the form $\Gamma_{1,j,k}$ for some two integers $j\ge 2$ and $k\ge 3$. 
The converse also holds, hence, $G$ contains a long unichord if and only if $G$ contains an induced subgraph isomorphic to a graph of the form $\Gamma_{1,j,k}$ with $j\ge 2$ and $k\ge 3$. 
Similarly, $G$ contains an induced long theta if and only if it contains an induced subgraph isomorphic to a graph of the form $\Gamma_{i,j,k}$ such that $i,j\ge 2$ and $k\ge 3$.

Since the house is isomorphic to the graph $\Gamma_{1,2,3}$, any subdivision of the house is isomorphic to the graph $\Gamma_{i,j,k}$ for some $i\ge 1$, $j\ge 2$, and $k\ge 3$. 
Hence, a graph $G$ contains an induced subdivision of the house if and only if $G$ contains an induced subgraph isomorphic to the graph $\Gamma_{i,j,k}$ for some $i\ge 1$, $j\ge 2$, $k\ge 3$. 
As shown above, the cases $i = 1$ and $i\ge 2$ correspond to the cases when $G$ contains a long unichord and an induced long theta, respectively.
\end{proof}

By~\Cref{thm:forbidden-subgraphs-subdivision-of-house}, we can determine whether a given graph contains an induced subdivision of the house by testing if it has a long unichord or an induced long theta.
The former problem has already been solved in the literature, as follows.

\begin{theorem}[Trotignon and Pham~\cite{trotignon2018chi}]\label{thm:recognition-long-unichord}
Deciding whether a given graph $G$ has a long unichord can be performed in time $\mathcal{O}(n^4m^2)$.
\end{theorem}

In order to test for an induced long theta, we modify an algorithm by Chudnovsky and Seymour~\cite{chudnovsky2010three} to determine whether a given graph $G$ contains an induced theta.
The key ingredient in the proof is an efficient solution to the \emph{three-in-a-tree problem}, which takes as input a graph $G$ and a set $X\subseteq V(G)$ with $|X| = 3$, and the task is to determine whether $G$ contains an induced subgraph $T$ such that $T$ is a tree and $X\subseteq V(T)$.
Chudnovsky and Seymour gave an algorithm for the three-in-a-tree problem running in time $\mathcal{O}(mn^2)$ (see~\cite{chudnovsky2010three}).
This time complexity was significantly improved by Lai, Lu, and Thorup~\cite{DBLP:conf/stoc/LaiLT20}, who gave an algorithm running in time $\mathcal{O}(m\log^2 n)$.

\begin{proposition}\label{thm:recognition-long-theta}
Deciding whether a given graph $G$ contains an induced long theta can be performed in time $\mathcal{O}(n^8m\log^2 n)$.
\end{proposition}

\begin{sloppypar}
\begin{proof}
Let $G$ be a graph with $n$ vertices and $m$ edges.
Enumerate all five-tuples \hbox{$(a,b,v_1,v_2,v_3)$} of distinct vertices such that vertex $a$ is adjacent to each of the vertices in $\{b,v_1,v_2\}$, vertex $b$ is adjacent to $v_3$, vertices in $\{v_1,v_2,v_3\}$ are pairwise nonadjacent, and vertices in $\{b,v_1,v_2\}$ are pairwise nonadjacent.
For each such five-tuple $(a,b,v_1,v_2,v_3)$, enumerate all subsets $X\subseteq V(G)$ such that $a$ and $b$ have no neighbours in $X$, and $v_1,v_2,v_3$ each have exactly one neighbour in $X$, and each member of $X$ is adjacent to at least one of $v_1,v_2,v_3$ (it follows that $|X|\le 3$).
For each such choice of $X$, let $G'$ be obtained from $G$ by deleting
$a, b$ and all vertices adjacent to one of $a,b$, $v_1,v_2,v_3$ except for the members of $\{v_1,v_2,v_3\}\cup X$ and test whether there is an induced tree in $G'$ containing all of $v_1,v_2,v_3$.
Indeed, we claim that $G$ contains an induced long theta if and only if there is some choice of $a,b,v_1,v_2,v_3$, and $X$ such that $G'$ contains an induced subgraph $T$ such that $T$ is a tree and $\{v_1,v_2,v_3\}\subseteq V(T)$. 
Assuming the claim, we have to run the three-in-a-tree algorithm at most $n^8$ times, and each one takes time $\mathcal{O}(m\log^2 n)$.

It remains to show the claim.
Suppose first that $G$ contains an induced long theta, that is, an induced subgraph $H$ isomorphic to a graph of the form $\Gamma_{i,j,k}$ such that $i,j\ge 2$ and $k\ge 3$.
Let $P^1$, $P^2$, and $P^3$ be the three edge-disjoint paths forming $\Gamma_{i,j,k}$ of lengths $i$, $j$, and $k$, respectively.
Let $a$ be a vertex of degree $3$ in $H$, let $v_1$ and $v_2$ be the neighbors of $a$ in $P^1$ and $P^2$, respectively, let $b$ be the neighbor of $a$ in $P^3$, and let $v_3$ be the neighbor of $b$ in $P^3$ other than $a$. 
For $i\in \{1,2,3\}$, let $x_i$ be the neighbor of $v_i$ in $P^i$ such that $X = \{x_1,x_2,x_3\}$ is disjoint from $\{a,b\}$.
Let $G'$ be the graph obtained from $G$ by deleting $a, b$ and all vertices adjacent to one of $a,b$, $v_1,v_2,v_3$ except for the members of $\{v_1,v_2,v_3\}\cup X$.
Then $T = H-\{a,b\}$ is an induced subgraph of $G'$ such that $T$ is a tree and $\{v_1,v_2,v_3\}\subseteq V(T)$. 

Conversely, suppose that there is some choice of $a,b,v_1,v_2,v_3$, and $X$ in $G$ such that $G'$ contains an induced subgraph $T$ such that $T$ is a tree and $\{v_1,v_2,v_3\}\subseteq V(T)$. 
We may assume that $T$ is an inclusion-minimal induced subtree of $G'$ such that $\{v_1,v_2,v_3\}\subseteq V(T)$. 
By the minimality of $T$, the tree $T$ has at most three leaves, and if it has three leaves, then the leaves are exactly the vertices $v_1$, $v_2$, and $v_3$.
If $T$ has only two leaves, then $T$ is a path from $v_p$ to $v_q$ having $v_r$ as an internal vertex, where $\{p,q,r\} = \{1,2,3\}$.
However, this means that $v_r$ two neighbors in $X$, a contradiction.
Hence, $T$ has exactly three leaves, namely $v_1$, $v_2$, and $v_3$, and, furthermore, the minimality of $T$ implies that $T$ is isomorphic to a subdivision of a claw.
It follows that the subgraph of $G$ induced by $V(T)\cup \{a,b\}$ is a theta; in fact, it is a long theta, since the path between the two vertices of degree three containing $b$ has length at least three.
\end{proof}
\end{sloppypar}

The following result implies \Cref{thm:recognition-house-subdivision}.

\begin{theorem} \label{thm:recognition-house-subdivision-improved}
Determining if a given graph $G$ contains the house as an induced topological minor can be done in time $\mathcal{O}(n^8m\log^2 n)$.
\end{theorem}

\begin{proof}
The result follows from~\cref{thm:forbidden-subgraphs-subdivision-of-house,thm:recognition-long-theta,thm:recognition-long-unichord}. 
\end{proof}

\subsection{Proof of~\cref{thm:recognition-house-minor}}
\label{sec:recognition-house-minor}

In this section we show that detecting if a graph contains the house as an induced minor can be done in polynomial time.
To reduce this problem to the induced topological minor case, we show (in \Cref{thm:characteriation-him-free}) that the existence of the house as the induced minor in a graph $G$ is equivalent to the existence of the house as an induced topological minor or an induced long twin wheel in $G$, that is, a graph obtained from a cycle of length at least five by replacing a vertex with a pair of adjacent vertices with the same closed neighborhoods.

First we show a useful result about graphs that do not contain the house as an induced topological minor (or, equivalently, graphs that do not contain any induced subdivision of the house). 
Given a graph $G$, a subgraph $H$ of $G$, and a vertex $v\in V(G)\setminus V(H)$, we say that $v$ is a \emph{pendant} of $H$ if $v$ is adjacent to a single vertex of $H$. 

\begin{lemma}\label{nbh-3cases} 
  Let $G$ be a graph that does not contain any induced subdivision of the house, $H$ be a hole in $G$ and $v\in V(G)\setminus V(H)$ be a vertex with a neighbor in $H$.
  Then one of the following is true: $v$ is pendant of $H$, or the neighbors of $v$ in $H$ are exactly three vertices that are consecutive, or $v$ is universal for $H$. 
  \end{lemma}
  \begin{proof}
      Let $G$, $H$, $v$ be as stated and let $v_1,\ldots, v_k$ be vertices of $H$ in a cyclic order. 
      If $v$ has only one neighbor in $H$, or is universal for $H$, we are done, so we may assume that $v$ has at least two neighbors in $H$ and at least one non-neighbor in $H$. 
      If $v$ has exactly two neighbors in $H$, then $V(H)\cup \{v\}$ induce a subdivision of the house in $G$. 
      It follows that $v$ has at least three neighbors in $H$. 
      Without loss of generality we may assume that $vv_1\in E(G)$ and $vv_k\notin E(G)$. 
      Let $v_i$ and $v_j$ be neighbors of $v$ in $H$ such that $1<i<j<k$ and no vertex $v_{\ell}$ with $i<\ell<j$ is adjacent to $v$.
      If $v_1,v_i,v_j$ are not consecutive, then $v$ and the vertices of the $v_1,v_j$-path in $H-v_k$ induce a subdivision of the house in $G$; a contradiction.
      Hence, $v_1,v_i,v_j$ are consecutive in $H$, that is, $i=2$ and $j=3$. 
      If there is a neighbor of $v$ in $G$ among the vertices $\{v_4,\ldots, v_{k-1}\}$, let $v_{\ell}$ be the one with the largest index. 
      Then $v$ and the vertices of the $v_2,v_\ell$-path in $H-v_3$ induce a subdivision of the house in $G$; a contradiction.
      It follows that such a vertex $v_\ell$ cannot exist and $v$ has exactly three consecutive neighbors in $H$.     
  \end{proof}
  
A \emph{long twin wheel} is a graph consisting of a hole $H$ of length at least $5$ and of another vertex $v$, called the \emph{center}, such that $v$ has degree three and the neighborhood of the center induces a connected graph in $H$.

\begin{theorem}\label{thm:characteriation-him-free}
Let $G$ be a graph. 
Then the following statements are equivalent.
\begin{enumerate}
\item\label{him-free} $G$ does not contain the house as an induced minor.
\item\label{his-free-tw-free} $G$ does not contain an induced subdivision of the house and does not contain an induced long  twin wheel. 
\end{enumerate}
\end{theorem}
\begin{proof}

First we prove that $\eqref{him-free}$ implies $\eqref{his-free-tw-free}$.
Let $G$ be a graph that does not contain the house as an induced minor.
Since an induced topological minor is a particular case of an induced minor, it follows that $G$ does not contain an induced subdivision of the house.
Suppose for a contradiction that $G$ contains a long twin wheel $H$ as an induced subgraph.
Let $v_0,v_1,\ldots, v_k$ be vertices of the hole in $H$, in a cyclic order, and let $v$ be the center of $H$. 
Up to symmetry, we may assume that $v_0,v_1,v_2$ are the neighbors of $v$ in $H$. 
Contracting all the edges of the hole in $H$, except those incident with $v_k$ and $v_2$ gives a house in $G$, a contradiction with the assumption that $G$ is house-induced-minor-free. 

Now we prove that $\eqref{his-free-tw-free}$ implies $\eqref{him-free}$.
Let $G$ be a graph that does not contain an induced subdivision of the house and does not contain an induced long twin wheel.
Let $H$ be a house obtained from the $5$-cycle with vertices $p,q,r,s,t$ in cyclic order by adding to it the chord $qt$ and assume that $G$ contains $H$ as an induced minor. 
Let $\{X_v\}_{v\in V(H)}$ be an induced minor model of $H$ in $G$. 
By \Cref{prop:simple-induced-minor-models-of-all-paths}, we may assume that the sets $X_s$, $X_r$ and  $X_p$ are all singletons and let $X_s=\{a\}$, $X_r=\{b\}$ and $X_p=\{c\}$. 
For simplicity, we denote sets $X_t$ and $X_q$ by $A$ and $B$, respectively.
Then $G$ contains two connected subsets $A$ and $B$ and vertices $a, b,c \in V(G)\setminus (A\cup B)$ such that $a$ and $b$ are adjacent, $a$ has a neighbor in $A$ but not in $B$, $b$ has a neighbor in $B$ but not in $A$, $c$ has a neighbor in both $A$ and $B$ but is not adjacent to any of $a$ and $b$, and there is an edge between $A$ and $B$. 

By assumption, there exists an edge in $G$ with endpoints in $A$ and $B$, so the subgraph of $G$ obtained from $G[A\cup B\cup \{a,b\}]$ by deletion of the edge $ab$ is connected. 
Hence, there is an induced $a,b$-path in $G[A\cup B\cup \{a,b\}]-ab$.
Note that every such path is of length at least~$3$. 
Let $\mathcal{P}$ be the set of all such induced $a,b$-paths in $G[A\cup B\cup \{a,b\}]-ab$.
Given $P\in \mathcal{P}$, let $H^P$ be the hole consisting of the edge $ab$ and of the path $P$,  let $P^A$ be a shortest path from $c$ to the hole $H^P$ such that all vertices of $P^A$ other than $c$ belong to $A$, and let $P^B$ be a shortest path from $c$ to the hole $H^P$ such that all vertices of $P^B$ other than $c$ belong to $B$. 
Finally, let $P\in \mathcal{P}$ be a path that minimizes the number $|E(H^P)|+|E(P^A)|+|E(P^B)|$.
In the rest of the proof we simply denote by $H$ the hole $H^P$.

Let $c_A$ and $c_B$ be vertices on $P^A$ and $P^B$, respectively, that have neighbors on $H$. 

Assume first that both $P^A$ and $P^B$ are of length one.
Then $c=c_A=c_B$ has at least two neighbors in $H$. 
As $G$ is a graph without an induced subdivision of the house, by \Cref{nbh-3cases} it follows that $c$ is either universal for $H$, or has exactly three consecutive neighbors in $H$. 
However, since $c$ is not adjacent to $a$ and $b$, it follows that $c$ has exactly three consecutive neighbors in $H. $
Thus, $H$ has at least five vertices and together with $c$ induces a long twin wheel, a contradiction to the definition of $G$.
Hence, up to symmetry, we may assume that $c\neq c_A$ and thus $P_A$ is of length at least two.

All internal vertices of $P_A$ belong to $A$, so $c_A$ is not adjacent to $b$. 
Then, by \Cref{nbh-3cases}, it follows that $c_A$ either has one neighbor in $H$ or it has exactly three consecutive neighbors in $H$. 
Assume first that $c_A$ has exactly three consecutive neighbors in $H$. 
If $|V(H)|\ge 5$, then $H$ together with $c_A$ induces a long twin wheel, a contradiction. 
This implies that $H$ is a hole on four vertices. 
Let $\{a,b,y,x\}$ be vertices in $H$, in cyclic order. 
Then defining $H'$ to be the hole induced by $\{a,b,y,c_A\}$ implies that the corresponding shortest paths from $c$ to $H'$ are defined as $Q^A=P^A-x$ and $Q^B=P^B$. 
Altogether we have that $|E(H')|+|E(Q^A)|+|E(Q^B)|<|E(H)|+|E(P^A)|+|E(P^B)|$, which contradicts the minimality of $P$. 
Hence, $c_A$ has exactly one neighbor on the hole $H$. 

By symmetry, the similar argumentation can be used to verify that $c_B$ (if not equal to $c$) has exactly one neighbor on $H$. 
Moreover, if $c=c_B$, then by \Cref{nbh-3cases} it follows that $c$ either has three consecutive neighbors on $H$, or $c$ has exactly one neighbor in $H$. 
In the first case it follows that $H$ and $c$ together induce a long twin wheel, a contradiction to definition of $G$. 
Hence, in either case, $c_B$ also has exactly one neighbor on the hole $H.$

If there is some edge $uv$ in $G$ such that $u$ is an internal vertex of the path $P_A$,  $u\neq c_A$, and $v$ belongs to the hole $H$, then the path consisting of the $c,u$-subpath of $P_A$ and of the edge $uv$ can be taken instead of the path $P_A$, contradicting the minimality condition. 
By symmetry, the same holds for $u$ being the internal vertex of the path $P_B$, distinct than $c_B$. 
It follows that the internal vertices of paths $P_A$ and $P_B$, distinct than $c_A$ and $c_B$ have no neighbors in $H$. 

If the sets of internal vertices of $P_A$ and $P_B$ are anticomplete to each other, then $H$ together with $P_A$ and $P_B$ induces a subdivision of the house, so we may assume there is some edge connecting the internal vertices of $P_A$ and $P_B.$
Let $x$ be the neighbor of $c_A$ in $H$, and let $y$ be the neighbor of $c_B$ in $H$. 
Note that $x\neq y$, since $x\in A$, $y\in B$. 
We know that there is an $x,y$-path in $G[V(P_A)\cup V( P_B)].$
Let $Q$ be a shortest such path.
The existence of the edge connecting the internal vertices of $P_A$ and $P_B$ ensures that $c\notin V(Q).$
Let $R$ be the $x,y$-path in the hole $H$ that contains vertices $a$ and $b$, and let $S$ be the $xy$-path in $H-a$.

Clearly, $Q$ is of length at least three, and at least one of the paths $R$ and $S$ is of length at least two, so the vertices of paths $R$, $S$, $Q$ induce an $\Gamma_{i,j,k}$ with $i\ge 1$, $j\ge 2$, $k\ge 3$. 
But then $\Gamma_{i,j,k}$ is an induced subdivision of the house in $G$, a contradiction. 
\end{proof}

The following result implies \cref{thm:recognition-house-minor}.

\begin{theorem}
Determining if a given graph $G$ contains the house as an induced minor can be done in time $\mathcal{O}(n^8m\log^2 n)$.
\end{theorem}

\begin{proof}
Let $G=(V,E)$ be a graph.
By \Cref{thm:characteriation-him-free}, in order to test if $G$ is house-induced-minor-free, it suffices to check whether $G$ contains an induced subdivision of the house and whether $G$ contains an induced long twin wheel. 
If any of these conditions is satisfied, then $G$ is not house-induced-minor-free; otherwise, it is.
The former condition can be tested in polynomial time by \Cref{thm:recognition-house-subdivision-improved}.
The latter condition can also be tested in polynomial time, as follows.
It is not difficult to verify that $G$ contains an induced long twin wheel if and only if there exist four vertices $a,b,c,d$ in $G$ such that the subgraph of $G$ induced by $\{a,b,c,d\}$ is isomorphic to the diamond, the vertices $a$ and $d$ are nonadjacent, and there exists an $a,d$-path in the graph obtained from $G$ by deleting from it the vertices in $(N[b]\cup N[c])\setminus \{a,d\}$ as well as all common neighbors of $a$ and $d$.
Checking this condition over all the four-tuples of vertices of $G$ can be done in time $\mathcal{O}(n^4m)$.
\end{proof}

\subsection{Detecting the butterfly as an induced minor}\label{subsec:butterfly}

Determining if a given graph $G$ contains the butterfly as an induced minor can also be done in polynomial time.
This is an immediate consequence of the following characterization of graphs containing the butterfly as an induced minor.

\begin{proposition}[Maël Dumas and Claire Hilaire, personal communication, 2024]\label{prop:butterfly}
Let $G$ be a graph.
Then, $G$ contains the butterfly as an induced minor if and only if there exists a set $X\subseteq V(G)$ inducing a subgraph isomorphic to $2P_2$ and a connected component $C$ of the graph $G-X$ such that every vertex in $X$ has a neighbor in $C$.
\end{proposition}

For completeness, we include a short proof of \Cref{prop:butterfly} based on~\cref{prop:simple-induced-minor-models-of-all-paths}.

\begin{proof}
Assume first there exists an set $X\subseteq V(G)$ inducing a subgraph isomorphic to $2P_2$ and a connected component $C$ of the graph $G-X$ such that every vertex in $X$ has a neighbor in $C$.
Deleting all vertices in $V(G)\setminus (X\cup V(C))$ and contracting all edges within $C$ results in the butterfly, showing that $G$ contains the butterfly as an induced minor.

Conversely, assume that the butterfly is an induced minor of $G$.
Fix a graph $H$ isomorphic to the butterfly, let $z$ the vertex of degree $4$ in $H$ and let $U = V(H)\setminus \{z\}$.
Then $U$ is the set of vertices of degree $2$ in $H$.
By \Cref{prop:simple-induced-minor-models-of-all-paths}, there exists an induced minor model $\{X_v\}_{v\in V(H)}$ of $H$ in $G$ such that $|X_u|=1$ for all $u\in U$. 
Let $X = \bigcup_{u\in U}X_u$.
Then, the subgraph of $G$ induced by $X$ is isomorphic to $2P_2$.
Furthermore, since the set $X_z$ induces a connected subgraph of $G-X$, there exists a connected component $C$ of the graph $G-X$ such that $X\subseteq V(C)$.
Since every vertex in $U$ is adjacent to $z$ in $H$, every vertex in $X$ is adjacent in $G$ to a vertex in $X_z$ and hence to a vertex in $C$. 
This completes the proof.
\end{proof}

\begin{theorem}
\label{thm:recognition-butterfly-minor}
Determining if a given graph $G$ contains the butterfly as an
induced minor can be done in time $\mathcal{O}(n^5m)$. 
\end{theorem}

\begin{proof}
Immediate from \Cref{prop:butterfly}. 
Indeed, given a graph $G = (V,E)$, we check all subsets $X\subseteq V(G)$ with $|X| = 4$. 
For each such subset inducing a $2K_2$, we compute the connected components $C$ of the graph $G-X$, for each of them we compute the neighborhood of $V(C)$ in $G$ and test if $X\subseteq N(V(C))$.
If one such pair $(X,C)$ is found, then $G$ contains the butterfly as an
induced minor, otherwise, it does not.
With a straightforward implementation, the time complexity of the algorithm is $\mathcal{O}(n^5m)$.
\end{proof}

\section{Conclusion}
\label{sec:conclusion}

In this work we characterized tame graph classes defined by a single forbidden induced minor or induced topological minor. 
A natural question for future research would be to obtain analogous characterizations for graph classes excluding a single graph as a minor or topological minor, or, more generally, for larger (but still finite) sets of forbidden structures, as suggested already in~\cite{gajarsky2022taming,gartland2023quasi}.

We also provided some partial answers to \Cref{que:MWIS-poly-for-planar-H,que:recognition-H-induced-minor-free,que:recognition-H-induced-topological-minor-free}.
In particular, \Cref{thm:tame-applications} implies that MWIS is solvable in polynomial time in the class of $H$-induced-minor-free graphs when $H$ is the house.
Besides a further investigation of the above questions, other challenging avenues for future research include the analogue of \Cref{que:MWIS-poly-for-planar-H} for graph classes excluding a single induced topological minor~$H$.
In this case, as \textsc{Independent Set} is \NP-complete for cubic planar graphs (see~\cite{zbMATH01843763}), $H$ would have to be planar and subcubic.

\paragraph{Acknowledgements}

We are grateful to Claire Hilaire and Nicolas Trotignon for useful discussions.
This work is partially supported by the Slovenian Research and Innovation Agency (I0-0035, research program P1-0285 and research projects J1-3001, J1-3002, J1-3003, J1-4008, and J1-4084, and a Young Researchers Grant), and by the research program CogniCom (0013103) at the University of Primorska.

\end{document}